\documentclass{psp}

\usepackage[T1]{fontenc}  
\usepackage{amsmath,amsfonts,amssymb} 
\usepackage{mathrsfs}          
\usepackage{dsfont}
\usepackage{tikz}					
\usepackage{subfigure}

\allowdisplaybreaks

\newcommand{\CC}{\mathds{C}}
\newcommand{\PP}{\mathds{P}}
\newcommand{\RR}{\mathds{R}}

\newcommand{\A}{\mathcal{A}}
\newcommand{\Q}{\mathcal{Q}}
\newcommand{\T}{\mathcal{T}}

\newcommand{\C}{\mathcal{C}}

\newcommand{\N}{\mathcal{N}}

\newcommand{\p}{\pi_1}
\newcommand{\inv}{^{-1}}
\renewcommand{\P}{\mathcal{P}}
\renewcommand{\epsilon}{\varepsilon}
\newcommand{\aff}{^{\text{aff}}}
\newcommand{\set}[1]{\left\{ #1 \right\}}
\newromanexpr\sgn{sgn}
\newromanexpr\HH{H}

\newcommand{\ncross}[2]{\draw[thick,cap=round](#2,#1)--(1+#2,#1);}
\newcommand{\ocross}[2]{\draw[thick,cap=round](#2,#1)--(1+#2,1+#1); \draw[thick,color=white,line width=6pt] (0.2+#2,1+#1)--(0.8+#2,#1); \draw[thick,cap=round](#2,1+#1)--(1+#2,#1);}
\newcommand{\ucross}[2]{\draw[thick,cap=round](#2,1+#1)--(1+#2,#1); \draw[thick,color=white,line width=6pt] (0.2+#2,#1)--(0.8+#2,1+#1); \draw[thick,cap=round](#2,#1)--(1+#2,1+#1);}
\newcommand{\rcross}[2]{\draw[thick,cap=round](#2,1+#1)--(1+#2,#1); \draw[thick,cap=round](#2,#1)--(1+#2,1+#1);}
\newcommand{\mcross}[3]{\foreach \k in {1,2,...,#2} {\draw[thick,cap=round](#3,-1+\k+#1)--(1+#3,-\k+#1+#2);} }

\newtheorem{thm}{Theorem}[section]
\newtheorem{propr}[thm]{Property}
\newtheorem{propo}[thm]{Proposition}

\newtheorem{cor}[thm]{Corollary}
\newtheorem{de}[thm]{Definition}
\newtheorem{ntc}[thm]{Notation}
                        
\newtheorem{ex}[thm]{Example}
\newtheorem{rmk}[thm]{Remark}

\begin{document}

\title{On complex line arrangements and their boundary manifolds}

\author[V. Florens, B. Guerville-Ball{\'e} and M.A. Marco Buzunariz]
{V. Florens\\
	LMA, UMR CNRS 5142, Universit{\'e} de Pau et des Pays de l'Adour \textup{64000} Pau, France\addressbreak
  e-maill\textup{: \texttt{vincent.florens@univ-pau.fr}}
\nextauthor B. Guerville-Ball{\'e}\\
	IJF, UMR 5582 CNRS-UJF, Universit{\'e} Grenoble Alpes \textup{38 000} Grenoble, France\addressbreak
  e-maill\textup{: \texttt{benoit.guerville-balle@math.cnrs.fr}}
\and\ M.A. Marco Buzunariz\\
	ICMAT: CSIC-Complutense-Autonoma-Carlos III, Departamento de Algebra \\ 
	Facultad de CC. Matematicas - Plaza de las Ciencias, \textup{3, 28040} Madrid, Spain\addressbreak
  e-maill\textup{: \texttt{mmarco@unizar.es}}}

\receivedline{Received \textup{6} February \textup{2014;}
              revised \textup{26} February \textup{2015}}

\maketitle

\begin{abstract}
	Let $\mathcal{A}$ be a line arrangement in the complex projective plane $\CC\PP^2$.
	We define and describe the inclusion map of the boundary manifold --the boundary of a close regular neighborhood of $\A$--  in the exterior of the arrangement. We obtain two explicit descriptions of the map induced on the fundamental groups. These computations provide a new minimal presentation of the fundamental group of the complement.
\end{abstract}


\section{Introduction}

Line arrangements are finite collections of complex lines in the projective space $\CC\PP^2$, that is,  plane algebraic curves whose irreducible components are all of degree one. The general study of discriminants of curves in $\CC\PP^2$ and their stratification leads to considering the homeomorphism type of the pair as a natural isotopy invariant. We refer to this invariant as the \emph{topology of the embedding}. O.~Zariski was the first to show that the \emph{combinatorial description} of a curve (degree of the components, local type of the singularities,...) is not enough to determine the topology. The case of line arrangements is quite motivating since lines are non-singular and two lines intersect at a single point: the combinatorial structure of an arrangement can  easily be encoded in the \emph{incidence graph}. However, S.~MacLane \cite{mcl:36}  showed that it does not determine the deformation class. Later G.~Rybnikov \cite{ry:98} showed that combinatorics does not determine even the topological type of the complement (see also \cite{accm,accm:03a}). This motivates the study of topological invariants such as  the fundamental group (and related: characters, Alexander invariants, characteristic varieties,...). 

 On the other hand, one may consider the \emph{boundary manifold} $B_\A$ of an arrangement $\A$, defined as the boundary of a closed regular neighborhood in $\CC\PP^2$.  It is a compact graph 3-manifold in the sense of F.~Waldhausen \cite{wal:67b}, whose topology is combinatorially determined~\cite{westlund,cosu:08}. In particular the graph structure is modeled by the incidence graph $\Gamma_\A$. Its fundamental group can be computed from this description, see for example~\cite{cosu:08}.

Our general aim is to study the inclusion map of the boundary manifold $B_\A$ in the exterior $E_\A$ of the arrangement in $\CC\PP^2$ and to give an explicit method to compute the map at the level of their fundamental groups. This is related to the work of E.~Hironaka \cite{eh:01} on complexified real arrangements, but the complex case requires a more careful study of generators of $\pi_1(B_\A) $, coming from cycles of the graph $\Gamma_\A$. From these computations, we derive a new minimal presentation of $\p(E_\A)$. 

Our main motivation is a series of papers (of joint works with E.~Artal~\cite{AFG,gue,gue:phd}), where applications of the map and its description are given.  We observe that the inclusion map captures some relevant information on the position of singularities that is not contained in the combinatorics. Indeed, in~\cite{AFG}, we use it to construct a new topological invariant of arrangements, see~\cite{gue,gue:phd} for illustrations and examples. Let us also mention that our method allows to complete the work of E.~Artal \cite{ea:toulouse} on the essential coordinate components of the characteristic varieties of an arrangement. It provides a crucial geometrical ingredient to compute the depth of any characters of the fundamental group (see in particular \cite{ea:toulouse} Section 5.4). This gives the only known way ---a geometrical way--- to compute this algebraic invariant of arrangements.
 
In Section 2, we recall the basics on combinatorics of arrangements. We construct the boundary manifold $B_\A$ from the incidence graph  $\Gamma_\A$ and give a presentation of its fundamental group. Section 3 is devoted to the complement $E_\A$ and the calculation of its fundamental group from the braided wiring diagram. In Section 4, we present the method to compute the inclusion map on fundamental groups. We obtain a description of the homotopy type of the exterior where the boundary manifold appears explicitly. In Section 5, we illustrate the method using MacLane's arrangement.

Along the different sections, the notions and computations are illustrated with the \emph{didactic} example described by the following equations:
\begin{equation*}
	\begin{array}{c}
		L_0=\left\{z=0\right\},\ L_1=\left\{ -(i + 2)x + (2i + 3)y = 0 \right\},\ L_2=\left\{ -x + (i + 2)y = 0 \right\}, \\
		L_3=\left\{ -x + 3y + iz = 0 \right\},\ L_4=\left\{ -x + (2i + 2)y = 0 \right\}.
	\end{array}
\end{equation*}


\section{The boundary manifold}\label{BoundaryManifold}

	We sometimes use both projective and affine points of view on arrangements. For a given arrangement $\A$ in $\CC\PP^2$ with $n+1$ lines, the line $L_0$ will denote an arbitrary choice of the line at infinity. The arrangement $\A-L_0$  in $\PP^2 -L_0 \simeq \CC^2$ is an affine arrangement with $n$ lines.
	 
The boundary manifold $B_\A$ is the boundary of a closed regular neighborhood of $\A$, which can be constructed as a sub-complex of a triangulation of $\PP^2$ --the closed star of $\A$ in the second barycentric subdivision. This is a compact, connected, oriented graph $3$-manifold, modeled on the incidence graph. In particular, it is combinatorially determined: any isomorphism of the incidence graph induces an isomorphism of the graph manifold,~\cite{eh:01}.


\subsection{Incidence graph}
\vspace{0.25cm}

Let $\A$ be an arrangement with set of singular points $\mathcal{Q}$. The incidence graph encodes the combinatorial information on $\A$, see \cite{orte:92} for details.
For $P\in\mathcal{Q}$, let us denote $\A_P=\{\ell\in\A\mid P\in\ell \}$. The number $m_P=\# A_P\geq 2$ is called the \emph{multiplicity} of~$P$.

\begin{de}\label{Incidence}
	 The \emph{incidence graph} $\Gamma_{\A}$ of $\A$ is a non-oriented bipartite graph where the set of vertices $V(\A)$ decomposes as $V_P(\A)\amalg V_L(\A)$, where
	\begin{equation*}
		V_P(\A)=\{v_P\mid P\in \mathcal{Q}\},\quad
		V_L(\A)=\{v_L\mid L\in \A\}.
	\end{equation*}
	The vertices of $V_P(\A)$ are called \emph{point-vertices} and those of $V_L(\A)$ are called \emph{line-vertices}.
	The edges of $\Gamma_{\A}$ join $v_L$ to $v_P$ if and only if $L\in \A_P$. They are denoted $e(L,P)$.
\end{de}

A morphism between incidence graphs is a morphism of graphs preserving the vertex labelings, which send elements of $V_P(\A)$ (resp. $V_L(\A)$) to elements
 of $V_P(\A)$  (resp. $V_L(\A)$).

The incidence graph of the didactic example is pictured in Figure \ref{DidacticIncidence}. 

\begin{figure}[!ht]	
	\begin{center}
	\begin{tikzpicture}[xscale=1.3,yscale=0.6]
	
		\node[draw] (L0) at (0,3) {$L_0$};
		
		\node[draw] (P01) at (1,6) {$P_{0,1}$};
		\node[draw] (P02) at (1,4) {$P_{0,2}$};
		\node[draw] (P04) at (1,2) {$P_{0,4}$};
		\node[draw] (P03) at (1,0) {$P_{0,3}$};
		
		\node[draw] (L3) at (2.5,0) {$L_3$};
		\node[draw] (L4) at (2.5,2) {$L_4$};
		\node[draw] (L2) at (2.5,4) {$L_2$};
		\node[draw] (L1) at (2.5,6) {$L_1$};
				
		\node[draw] (P124) at (6,6) {$P_{1,2,4}$};
		\node[draw] (P13) at (6,4) {$P_{1,3}$};
		\node[draw] (P23) at (6,2) {$P_{2,3}$};		
		\node[draw] (P34) at (6,0) {$P_{3,4}$};
		
		\draw[thick, line width=1.5pt, cap=round] (L0) -- (P01);
		\draw[thick, line width=1.5pt, cap=round] (L0) -- (P02);
		\draw[thick, line width=1.5pt, cap=round] (L0) -- (P03);
		\draw[thick, line width=1.5pt, cap=round] (L0) -- (P04);
		
		\draw[thick, line width=1.5pt, cap=round] (P01) -- (L1);
		\draw[thick, line width=1.5pt, cap=round] (P02) -- (L2);
		\draw[thick, line width=1.5pt, cap=round] (P03) -- (L3);
		\draw[thick, line width=1.5pt, cap=round] (P04) -- (L4);
		
		\draw[thick, line width=1.5pt, cap=round] (L1) -- (P124);
		\draw[thick, line width=1.5pt, cap=round] (L1) -- (P13);
		\draw[thick, line width=1.5pt, cap=round] (L2) -- (P23);
		\draw[thick, line width=1.5pt, cap=round] (L4) -- (P34);
		
		\draw (L2) -- (P124);
		\draw (L4) -- (P124);
		\draw (L3) -- (P13);
		\draw (L3) -- (P23);
		\draw (L3) -- (P34);		
	\end{tikzpicture}
\end{center}
	\caption{Incidence graph of the didactic example}
	\label{DidacticIncidence}
\end{figure}
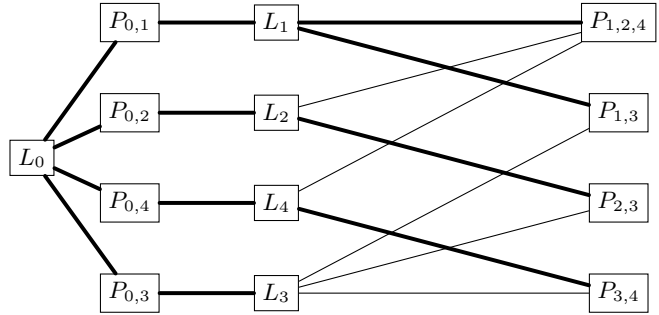


\subsection{Construction of $B_\A$}\label{BoundaryManifoldConstruction}
\vspace{0.25cm}

Let $U$ be a compact regular neighborhood of $\A$. We recall that the boundary manifold $B_\A$ can be defined as the boundary of $U$. This manifold $B_\A$ is combinatorially determined and can be computed from the incidence graph $\Gamma_\A$ as follows:

For every singular point $P\in\mathcal{Q}$ of $\mathcal{A}$, consider a $4$-ball $\mathbb{B}_P$ of radius $\eta$, centered in $P$. Let $\mathcal{S}_P = \partial   (\mathbb{B}_P) \backslash \mathbf{T}$, where $\mathbf{T}$ is an open regular neighborhood of the link $L_P=\left( \partial\mathbb{B}_P \cap \mathcal{A} \right)$. The boundary of $\mathcal{S}_P$ is a union of disjoint tori $T$ indexed by the lines $L_i$ passing through $P$, and $T_L=\left(L\cap\partial \mathbb{B}_P\right) \times S^1$.

\begin{de}\label{Def_MerLongTl}
	Let $P \in \mathcal{Q}$ and $L \in \A$ be such that $P \in L$. The \emph{meridian} $m_L$ and the \emph{longitude} $l_L$ of the torus $T_L$ are the pair of oriented simple closed curves in $T_L\subset\partial \overline{\mathbf{T}}$ which are determined up to isotopy by the homology and linking relations:
	\begin{equation*}
		\begin{array}{c}
			m_L\thicksim 0,\quad l_L\thicksim (L\cap\overline{\mathbf{T}})\quad \text{ in }\HH_1(\overline{\mathbf{T}}) ;\\[6pt]
			\ell(m_L, L\cap\overline{\mathbf{T}})=1,\quad \ell(l_L,L\cap\overline{\mathbf{T}})=0,
		\end{array}		
	\end{equation*}
	where $\ell(\cdot,\cdot)$ denotes the linking number in $\partial\mathbb{B}_P\simeq S^3$.
\end{de}

Consider the surface:
	\begin{equation*}
		F = \mathcal{A} \setminus  \coprod\limits_{P\in\mathcal{Q}} \left(\mathcal{A} \cap \overset{\circ}{\mathbb{B}}_P\right).
	\end{equation*}	
	It is obtained by removing of $\A$ open discs from the  $\mathbb{B}_P$'s. One sees that $F$ is a union $\coprod\limits_{i=0}\limits^n F_i$ where each $F_i$ corresponds to the line $L_i$ of $\mathcal{A}$. Let $\mathcal{N}_i = F_i \times S^1$ whose boundary is a union of disjoint tori $T$ indexed by the points $P \in \P \cap L_i$. 
	
	Let $D$ be a generic line (i.e. for all $P\in\Q$, $P\notin D$), and consider $D$ as the line at infinity. We decompose $\mathcal{N}_i$ in the solid torus $\mathcal{N}_i \cap T(D)=\mathbf{T}_i^\infty$, where $T(D)$ is regular neighborhood of $D$, and the affine part $\mathcal{N}_i\aff$ defined as the closure of $\mathcal{N}_i\setminus \mathbf{T}_i^\infty$. Viewed as the affine part, $\mathcal{N}_i\aff$ admits a natural trivialization in the affine space $\CC\PP^2\setminus D$, and we choose a section $s$ of $\mathcal{N}_i\aff$. 
	
\begin{de}\label{Def_MerLongTp}
	Let $L_i \in \A$ and $P \in \P$ be such that $P\in L_i$. The \emph{longitude} $l_P$ of the torus $T_P$ is the intersection of the section $s$ with $T_P$. The \emph{meridian} $m_P$ of $T_P$ is the class in $\HH_1(T_P)$ of $\set{*}\times S^1$.
\end{de}	
	
	\begin{rmk}
		To reconstruct $\mathcal{N}_i$ from $\mathcal{N}_i\aff$ and $\mathbf{T}_i^\infty$, we glue the boundary component of $\mathcal{N}_i\aff$ different of the $T_P$'s with $\partial \mathbf{T}_i^\infty$. The gluing is done identifying the intersection of the section in this component with the sum of a longitude of $\partial \mathbf{T}_i^\infty$ (i.e. a curve in  $\partial \mathbf{T}_i^\infty$ homologically equivalent to $L_i \cap \mathbf{T}_i^\infty$ in $\mathbf{T}_i^\infty$) and a meridian; and the meridian with a fiber of the $S^1$-fibration of $T_i^\infty$.
	\end{rmk}
	
	For each edge $e(L_i,P)$ of $\Gamma_\A$, glue $\mathcal{S}_P$ with $\mathcal{N}_i$ along $T_{L_i}$ and $T_P$ identifying meridian with meridian, and longitude with longitude. The manifold then obtained is the boundary manifold of $\A$. From this construction of $B_\A$, we deduce its structure as a graph manifold.
	
\begin{propo}[\cite{westlund,eh:01}]
	Let $\A$ be a complex line arrangement. The boundary manifold $B_\A$ is a graph manifold over the incidence graph $\Gamma_\A$.
\end{propo}

\begin{rmk}
	Our construction is different from the one in~\cite{westlund} and~\cite{cosu:08} which uses the blow-up $\hat{\A}$ of $\A$ and plumbing graph as defined by W.D.~Neumann in~\cite{neu:81}. Since the incidence graph of $\A$ and the dual graph of $\hat{\A}$ are equivalent, a result of F.~Waldhausen~\cite{wal:67b} shows that the corresponding manifolds are homeomorphic. With elementary computations, though the gluings are described differently, one may show that the two constructions coincide. The two constructions give different presentations of the fundamental group of the boundary manifold.
\end{rmk}

\begin{cor}\label{Combinatorics_Boundary}
	The boundary manifold of a complex line arrangement depends only on the combinatorics of the arrangement.
\end{cor}

\begin{proof*}
	The plumbing used to construct $B_\A$ as a graph manifold over $\Gamma_\A$ is combinatorial. The equivalence of $\Gamma_\A$ and the combinatorics of $\A$ lead to the result.
\end{proof*}


\subsection{Fundamental group of $B_\A$}\label{BoundaryManifoldPresentation}
\vspace{0.25cm}
 The fundamental group of $B_\A$ is the group associated to the incidence graph, see~\cite{westlund,cosu:08}. Two types of generators naturally appear: the \emph{meridians} of the lines and the \emph{cycles} related to the graph.

\begin{de}\label{Definition_Meridian}
	Let $L$ be a line in $\CC\PP^2$, and $b$ be a point in $\CC\PP^2\setminus\A$. A homotopy class $\alpha\in\pi_1(\CC\PP^2\setminus\A,b)$ is a \emph{meridian of $L$} if $\alpha$ has a representative $\delta$ constructed as follows: 
	\begin{itemize}
		\item there is a smooth complex analytic disc $\Delta\subset\CC\PP^2$ transverse to $L$ at a smooth point of $\A$ and such that $\Delta\cap L=\left\{b'\right\}\subset L$, and pick out a point $b''\in\partial \Delta$ .
		\item there is a path $a$ in $\CC\PP^2\setminus\A$ from $b$ to $b''\in\partial\Delta$;
		\item $\delta=a\inv\cdot\beta\cdot a$, where $\beta$ is the closed path based in $b'$ given by $\partial \Delta$ (in the positive direction).		
	\end{itemize}
\end{de}

Choose arbitrarily a line $L_0$ of the arrangement. Note that a meridian of $L_0$ is the product of the inverse of some meridians of the lines $L_1,\cdots,L_n$, in $E_\A$. Let $\P$ be the set of singular points of the affine arrangement $\A\aff=\A\setminus L_0$. We assume that $\A$ is ordered. In Sub-section~\ref{BraidedWiringDiagram}, a particular order will be fixed.

\begin{de}
	A \emph{cycle} of the incidence graph $\Gamma_\A$ is an element of $\p(\Gamma_\A,v_{L_0})$.
\end{de}

\begin{rmk}
	The group $\p(\Gamma_\A,v_{L_0})$ is a free group on $b_1(\Gamma_\A)$ generators.
\end{rmk}

	We construct a generating system $\mathscr{E}$ of cycles of $\Gamma_\A$  as follows. Let $\T$ be the maximal tree of $\Gamma_\A$ containing the following edges:
\begin{itemize}
	\item $e(L,P)$ for all $P\in L_0$, and $L\in\A$;
	\item $e(L_{\nu(P)},P)$ for all $P\in \P$ and  $\nu(P)=\min\{\ j\ \mid\ L_i\in\A_P\ \}$.
\end{itemize}

\begin{rmk}
	Up to the choice of an order on $\A$, this maximal tree is uniquely determined.
\end{rmk}

An edge in $\Gamma_\A \setminus \T$ is of the form $e(L_j,P)$, with $P\in\P$ and $L_j\in\A \setminus L_0$. By definition of a maximal tree, there exists a unique path $\lambda_{P,j}$ in $\T$ joining $v_{P}$ and $v_{L_j}$. The unique cycle of $\Gamma_\A$ containing the three line-vertices $v_{L_0}$, $v_{L_{\nu(P)}}$ and $v_{L_j}$, and no other line-vertex, is denoted by:
\begin{equation*}
	\xi_{\nu(P),j}=\lambda_{P,j} \cup e(L_j,P).
\end{equation*}
Let $\mathscr{E}$ be the set of cycles of $\Gamma_A$ of the form $\xi_{s,t}$. To each $\xi_{s,t}$ in $\mathscr{E}$ will correspond a cycle of $\pi_1( B_\A,X_0)$ (where $X_0\in\N_0$), that we denote $\mathfrak{e}_{s,t}$. 

\begin{ntc}
	We denote $[a_1,\cdots,a_m]$  the equality of all the cyclic permutations 
	\begin{equation*}
		a_1\cdots a_m=a_2\cdots a_m a_1=\cdots = a_m a_1\cdots a_{m-1}.
	\end{equation*}
\end{ntc}

For $i=0,\cdots,n$, let $\alpha_i$ be a meridian of $L_i$ contained in the boundary of a regular neighborhood of $L_0$, and for $\xi_{s,t}\in\mathscr{E}$, let $\mathfrak{e}_{s,t}$ be a non trivial cycle contained in $\left(\bigcup\limits_{v_P\in\xi_{s,t}} \mathcal{S}_P \right) \cup \left(\bigcup\limits_{v_L\in\xi_{s,t}} \N_L\right)$, coming from the gluing over the edge $e(L_t,P)$ where $P$ is $ L_s \cap L_t$. We assume that $(s,t)\neq (s',t')$ is equivalent to $ \mathfrak{e}_{s,t}\cap \mathfrak{e}_{s',t'}=X_0$.

\begin{propo}\label{BoundaryPresentation} 
 Let $\alpha_i$ and $\mathfrak{e}_{s,t}$ be as previously defined. For any singular point $P=P_{i_1,\cdots,i_m}$ with multiplicity $m$ and $i_1=\nu(P)$, let 
\begin{equation*}
	\mathcal{R}_P = [\alpha_{i_m}^{c_{i_m}},\cdots,\alpha_{i_2}^{c_{i_2}},\alpha_{i_1}], \hbox{ where } c_{i_j}=\mathfrak{e}_{i_1,i_j} \hbox{ for all } j=2,\cdots,m.
\end{equation*}
The fundamental group of the boundary manifold $B_\A$ admits the following presentation:
\begin{equation*} 	
	\p(B_\A,X_0)=\langle\alpha_0,\alpha_1,\cdots,\alpha_n,\mathfrak{e}_{s_1,t_1},\cdots,\mathfrak{e}_{s_l,t_l}\ |\ \bigcup\limits_{P\in\P} \mathcal{R}_P \rangle. 
\end{equation*}
\end{propo}
It is worth noticing that the  $\mathfrak{e}_{s,t}$ are not uniquely defined (see details in the proof).

\begin{proof*} 
	Consider $P\in\mathcal{Q}$. Assume that $P=P_{i_1,\cdots,i_m}$. Let  $y_{P,i_1},\dots, y_{P,i_m}$ be the 'local' meridians of the line $L_i$ in $\partial\mathbb{B}_P$. We have the following presentation of $\p(\mathcal{S}_P)$:
	\begin{equation*}
		\p( \mathcal{S}_P )=\langle y_{P,i_1},\cdots,y_{P,i_m}\ \mid\ [y_{P,i_m},\cdots,y_{P,i_1}] \rangle.
	\end{equation*}
	Remark that, according with Definition~\ref{Def_MerLongTl}, $y_{P,i_j}$ is a meridian of $T_{i_j}$ and a longitude is the product of the other $y_{P,i_k}$.
	
	Consider $k\in\set{0,\cdots,n}$. Let $\mathcal{Q}\cap L_k=\left\{ P_{k_1},\cdots,P_{k_l} \right\}$. Let $g_{k,k_i}$ be the image of a meridian in $F_k$ around $P_{k_i}$, viewed in $F_k\times\set{1}\subset~\mathcal{N}_k$, and $\alpha_k\in\p(\mathcal{N}_k)$ a meridian of $L_k$ contained in a regular neighborhood of $L_0$. We have the following presentation of $\p(\mathcal{N}_k)$:
	\begin{equation*}
		\p(\mathcal{N}_k)=\langle g_{k,k_1},\cdots,g_{k,k_l}, \alpha_k\ \mid\ \forall i\in\left\{1,\cdots,l\right\},\ \alpha_k\inv \cdot g_{k,k_i} \cdot \alpha_k = g_{k,k_i} \rangle.
	\end{equation*}
	Remark that according with Definition~\ref{Def_MerLongTp}, $g_{k,k_i}$ is a longitude and $\alpha_k$ is a meridian of $T_{P_i}$.
	
	As a first step, we only glue the $\mathcal{N}_k$'s and the $\mathcal{S}_P$'s over the edges of $\T$. To do this, we use Seifert-Van Kampen's Theorem, and we consider a contractible set $\Theta$ homeomorphic to $\T$ and joining the base points of the $\mathcal{N}_k$'s and the $\mathcal{S}_P$'s. The fundamental group of $B_\A$ is computed relative to $\Theta$.

	As a second step, we glue over the edges of $\Gamma_\A-\T$ (or equivalently the elements of~$\mathscr{E}$). Then we use Seifert-Van Kampen's Theorem and HNN-extension; and we denote by $\mathfrak{e}_{s,t}$ the cycle coming from the glue due to the edge $e(L_t,P)$, with $P=L_s \cap L_t$.
	
	Note that if $P_i\in L_j$, then the meridian $\alpha_j$ is identified with $y_{P_i,j}$ and $g_{j,i}$ is identified with the product of generators of $\p(\mathcal{S}_{P_i})$ not equal to $y_{P_i,j}$. After first doing a simplification, we obtain the following presentation of the fundamental group of the boundary manifold:
	\begin{equation*}
		\p(B_\A)=\langle\alpha_0,\alpha_1,\cdots,\alpha_n,\mathfrak{e}_{s_1,t_1},\cdots,\mathfrak{e}_{s_l,t_l}\ |\ \bigcup\limits_{P\in\mathcal{Q}} \mathcal{R}_P \rangle. 
	\end{equation*}
	
	By construction of $\T$, the relations $\mathcal{R}_P$, for $P\in L_0$, are of the form:
	\begin{equation*}
		[\alpha_{0_m},\alpha_{0_{m-1}},\cdots,\alpha_{0_2},\alpha_0]. 
	\end{equation*}
	But, the presentation of $\p(\mathcal{N}_0)$ implies that $\alpha_0$ commutes with the $g_{0,i}$ and by identification it commutes with $\alpha_i$, for $i\in\set{1,\cdots,n}$. Which implies that for all $P\in L_0$, the relations $\mathcal{R}_P$ are trivial.
\end{proof*}

\begin{ex}
	The fundamental group of the didactic example boundary manifold is:
	\begin{multline*}
		<\alpha_0,\alpha_1,\alpha_2,\alpha_3,\alpha_4,\epsilon_{1,2},\epsilon_{1,3},\epsilon_{1,4},\epsilon_{2,3},\epsilon_{3,4}\ \mid \\
		[\alpha_4^{\epsilon_{1,4}},\alpha_2^{\epsilon_{1,2}},\alpha_1]
		,\ [\alpha_3^{\epsilon_{1,3}},\alpha_1]
		,\ [\alpha_3^{\epsilon_{3,4}},\alpha_4] 
		,\ [\alpha_3^{\epsilon_{2,3}},\alpha_2]>.
	\end{multline*}
\end{ex}


\section{The complement}


Let $E_\A$ be the complement of a tubular neighborhood of $\A$. As before, we choose an arbitrary line $L_0\in\A$, and let $\CC^2$ be $\CC\PP^2-L_0$.

\subsection{Braided wiring diagrams}\label{BraidedWiringDiagram}
\vspace{0.25cm}
 
Consider a linear projection $\pi : \CC^2 \rightarrow \CC$, \emph{generic} in the sense that: 
\begin{itemize}
	\item For all $i\in\left\{1,\dots,n\right\}$, the restriction of $\pi_{|L_i}$ is a homeomorphism.
	\item Each multiple point lie in a different fiber of $\pi$.
\end{itemize}

We suppose that the points $x_i=\pi(P_i)$ have distinct real parts, and that we can order the points of $\pi(\P)$ by increasing real parts, so that $Re(x_1)<Re(x_2)\cdots<Re(x_k)$. A~smooth path $\gamma:[0,1]\rightarrow \CC$ emanating from $x_0$ with $Re(x_0)<Re(x_1)$, passing through $x_1,\cdots,x_k$ in order, and horizontal in a neighborhood of each $x_i$ is said to be \emph{admissible}.
 
\begin{de}{\label{Wiring}}
	The \emph{braided wiring diagram} associated to an admissible path $\gamma$ is defined by:
	\begin{equation*}
		W_{\A}=\left\{(x,y)\in \A \mid\ \exists t\in [0,1],\ p(x,y)=\gamma(t)\ \right\}.
	\end{equation*}
	The trace $\omega_i=W_\A\cap L_i$ is called the \emph{wire} associated to the line
$L_i$.
\end{de}

	Note that if $\A$ is a real complexified arrangement, then $\gamma=[x_0-\eta,x_k+\eta]\subset \RR$; and $W_\A\simeq \A\cap\RR^2$.

\begin{rmk}
	~\\
	i) The braided wiring diagram depends on the path $\gamma$, and on the projection $\pi$.\\
	ii) The set of singular points $\P$ is contained in $W_\A$.
\end{rmk}

We re-index the lines $L_1,\cdots,L_n$ such that:
\begin{equation*}
	I_i < I_j\ \Longleftrightarrow\ i < j,
\end{equation*}
where $I_i=Im(L_i\cap \pi\inv(x_0))$. On the representation described bellow of the braided wiring diagram, this re-indexation implies that the lines are ordered at the left of the diagram from the top to bottom. This fixes an order on $\A$.

Since the $x$ coordinates of the points of $W_\A$ are parametrized by
$\gamma$, the wiring diagram can be seen as a one dimensional object inside
$\RR^3\simeq [0,1] \times \CC$. Consider its image by a generic projection
$\gamma([0,1])\times \CC \rightarrow \RR^2$. If we take a plane projection of this
diagram (assume, for example, that it is in the direction of the vector
$(0,0,1)$ -that is, in the direction of the imaginary axis of the fibre-), we
obtain a planar graph. Observe that there are nodes corresponding to the image
of actual nodes in the wiring diagram in $\RR^3$ (that is, to a singular point of
the arrangement). Other nodes appear from the projection of undergoing and
overgoing branches of the wiring diagram in $\RR^3$. The two types of nodes are
called by W.~Arvola \emph{actual and virtual crossing}.

If we represent the virtual crossings in the same way that they are represented as in the case of braid diagrams, we obtain a schematic representation of the wiring diagram as in Figure \ref{ExempleWiring}. From now on, we will refer to this representation as the wiring diagram itself.
By genericity, we  assume that two crossings (actual or virtual) do not lie on
the same vertical line.

It is worth noticing that from the braided wiring diagram, one may extract the braid monodromy of $\A$, related to the generic projection $\pi$. The local equation of a multiple point is of the form $y^m - x^m$, where $m$ is the multiplicity, and the corresponding local monodromy is a full twist in the braid group with $m$ strands.


\subsection{Fundamental group of the complement} \label{color}
\vspace{0.25cm}

We recall briefly the method due to W.~Arvola \cite{arvola} to obtain a presentation of the fundamental group of the complement from a braided wiring diagram $W_\A$.
The algorithm goes as follows: start from the left of the diagram, assigning a generator $\alpha_i$ to each strand. Then follow the diagram from the left to the right, assigning a new word to the strands going through each crossing. The rules for this new assignation are given in Figure~\ref{Arvola}, where the $a_i$'s are words in the $\alpha_i$'s.

\begin{figure}[!ht]
	\begin{tikzpicture}
	\begin{scope}[shift={(0,0)},xscale=4]
		\mcross{0}{4}{0}
		\node[left] at (0,0) {$a_m$};
		\node[left] at (0,0.9) {$a_{m-1}$};
		\node[left] at (0,1.55) {$\vdots$};
		\node[left] at (0,2) {$a_2$};
		\node[left] at (0,3) {$a_1$};
		\node[right] at (1,0) {$a_1$};
		\node[right] at (1,0.9) {$(a_2)^{a_1}$};
		\node[right] at (1,1.55) {$\vdots$};
		\node[right] at (1,2.1) {$(a_{m-1})^{a_{m-2} \cdots a_1} \equiv (a_{m-1})^{a_m\inv}$};
		\node[right] at (1,3) {$(a_m)^{a_{m-1}\cdots a_1}\equiv a_m $};
		\node at (0.5,-0.5) {real crossing};
		\node at (0.5,-1) {(A)};
	\end{scope}
	\begin{scope}[shift={(-2,-4.5)},scale=2]
		\ocross{0.25}{0}
		\node[left] at (0,0.25) {$a_j$};
		\node[left] at (0,1.25) {$a_i$};
		\node[right] at (1,0.25) {$a_i$};
		\node[right] at (1,1.25) {$a_j^{a_i}$};
		\node at (0.5,-0.25) {positive virtual crossing};
		\node at (0.5,-0.5) {(B)};						
	\end{scope}
	\begin{scope}[shift={(4,-4.5)},scale=2]
		\ucross{0.25}{0}
		\node[left] at (0,0.25) {$a_j$};
		\node[left] at (0,1.25) {$a_i$};
		\node[right] at (1,0.25) {$a_i^{a_j\inv}$};
		\node[right] at (1,1.25) {$a_j$};		
		\node at (0.5,-0.25) {negative virtual crossing};
		\node at (0.5,-0.5) {(C)};
	\end{scope}
\end{tikzpicture}
	\caption{Computation of Arvola's words}
	\label{Arvola} 
\end{figure}

The notation in Figure~\ref{Arvola} is $a_i^{a_j}=a_j\inv a_i a_j$.

For each actual crossing $P$ that corresponds to a singular point of $\A$, suppose that the strands are labeled with the words $a_1,\ldots,a_m$ with respect to their order in the
diagram at this point $P$, from top to bottom, where $m=m_P$ is the multiplicity of $P$.
Then the following relations are added to the presentation:
\begin{equation*}
R_P = [a_{m},\ldots,a_1] = \{a_{m} \cdots a_1 = a_1 a_{m} \cdots a_2=\cdots =a_{m-1}\cdots a_1 a_m\}.
\end{equation*}

They correspond to the action of a half-twist on the free group, whereas the
action of a virtual crossing is given by the corresponding braid.
 
\begin{figure}[!ht]
	\begin{tikzpicture}
	\begin{scope}[shift={(0,-5)},x=25pt,y=30pt]

		\node at (0,3) [below,left] {1};
		\node at (0,2) [below,left] {2};
		\node at (0,1) [below,left] {3};		
		\node at (0,0) [below,left] {4};

		\ncross{0}{0};\ncross{1}{0};\ncross{2}{0};\ncross{3}{0};
		\ocross{0}{1};\ncross{2}{1};\ncross{3}{1};
		\ncross{0}{2};\mcross{1}{3}{2};
		\ucross{0}{3};\ncross{2}{3};\ncross{3}{3};
		\ncross{0}{4};\ucross{1}{4};\ncross{3}{4};
		\ncross{0}{5};\ocross{1}{5};\ncross{3}{5};
		\ocross{0}{6};\ncross{2}{6};\ncross{3}{6};
		\rcross{0}{7};\ncross{2}{7};\ncross{3}{7};
		\ncross{0}{8};\ucross{1}{8};\ncross{3}{8};
		\ncross{0}{9};\ncross{1}{9};\ucross{2}{9};
		\ncross{0}{10};\ncross{1}{10};\rcross{2}{10};
		\ncross{0}{11};\ocross{1}{11};\ncross{3}{11};
		\ncross{0}{12};\rcross{1}{12};\ncross{3}{12};
		
		\node at (0.5,3) [above] {$\alpha_1$};
		\node at (0.5,2) [above] {$\alpha_2$};
		\node at (0.5,1) [above] {$\alpha_3$};
		\node at (0.5,0) [above] {$\alpha_4$};
		
		\node at (2.3,1.1) [right, below] {$\alpha_4^{\alpha_3}$};
		\node at (2.5,0) [below] {$\alpha_3$};
		\node at (5,0) [below] {$\alpha_1^{\alpha_3\inv}$};
		\node at (3,1) [right] {$\alpha_1$};
		\node at (3.3,2) [above] {$\alpha_2^{\alpha_1}$};
		\node at (4,3) [above] {$\alpha_4^{\alpha_3}$};
		\node at (5.2,1.1) [below] {$\alpha_2^{\alpha_1\alpha_3\inv}$};
		\node at (5,2) [above] {$\alpha_3$};
		\node at (7,2) [above] {$\alpha_2^{\alpha_1}$};
		\node at (7,0) [below] {$\alpha_3$};
		\node at (7,1) [above] {$\alpha_1$};
		\node at (8.2,1.2) [left] {$\alpha_3$};
		\node at (9,0) [below] {$\alpha_1$};
		\node at (10,1.1) [below] {$\alpha_2^{\alpha_1\alpha_3\inv}$};
		\node at (10,2.1) [below] {$\alpha_4$};
		\node at (10,3) [above] {$\alpha_3$};
		\node at (12,3) [above] {$\alpha_4$};
		\node at (10.8,2.2) [right] {$\alpha_3$};
		\node at (12,1) [below] {$\alpha_3$};
		\node at (12.2,2) [above] {$\alpha_2^{\alpha_1}$};
		
	\end{scope}
\end{tikzpicture}
	\caption{The braided wiring diagram of the affine part of the didactic example}
	\label{ExempleWiring}
\end{figure}

\begin{thm}[Arvola \cite{arvola}]\label{ThmArvola}
	For $i=0,\dots,n$, let $\alpha_i$ be the meridians of the lines $L_i$. The fundamental group of the exterior of $\A$ admits the following presentation
	\begin{equation*}
		\p(E_\A)=< \alpha_1,\cdots,\alpha_n\ \mid\ \bigcup\limits_{P} R_P >,
	\end{equation*}
	where $P$ ranges over all the actual crossings of the wiring diagram $W_\A$.	
\end{thm}

\begin{ex}
	The braided wiring diagram of the didactic example is pictured in Figure \ref{ExempleWiring}. Its fundamental group is :
	\begin{equation*}
		<\alpha_1,\alpha_2,\alpha_3,\alpha_4\ \mid 
		\ [\alpha_4^{\alpha_3},\alpha_2,\alpha_1],
		\ [\alpha_3,\alpha_1],
		\ [\alpha_4,\alpha_3],
		\ [\alpha_3,\alpha_2^{\alpha_1}]>.
	\end{equation*}
\end{ex}


 \section{The inclusion map}\label{Inclusion_Map}

The main result of the paper is a complete description of the map induced on the fundamental groups, by the inclusion of the boundary in the exterior of an arrangement $\A$. The computation is done in two main steps.

Let $W_{\A}$ be the wiring diagram
  associated to the choice of a generic projection $\pi$ and an admissible path $\gamma$. 
We start by choosing a generating system $\mathscr{E}= \left\{\xi_{s,t}\right\}$ of cycles of the incidence graph $\Gamma_\A$. These cycles can be directly seen in $W_{\A}$,
 since it contains all the singular points and the vertices of $\Gamma_\A$ can be identified with their corresponding wires between two singular points.
Then, the first step is to "push" each cycle $\xi_{s,t}$ from $W_\A$ to the boundary manifold $B_\A$. The procedure is described in Section \ref{bound}, 
 and gives an explicit family $\left\{\epsilon_{s,t}\right\}$ of $\pi_1(B_\A)$, indexed by $\mathscr{E}$. This family, with the set of meridians of
 the lines,  generates $\pi_1(B_\A)$.  
The second step is to compute the images of these generators $\epsilon_{s,t}$ by the inclusion map. We use an ad hoc Arvola's algorithm to make the computations directly from
 $W_\A$, see Section \ref{inclusion}. Then the map is described in Theorem \ref{MainResult}.
 
 In Section \ref{sequ}, we examine the kernel of the map; this provides an exact sequence involving $\pi_1(B_\A)$ and $\pi_1(E_\A)$, see Theorem \ref{sequence}. We deduce in Section \ref{homotopy} a presentation of $\pi_1(E_\A)$ where the generators of $\p(B_\A)$ appear explicitly. This presentation defines a complex whose homotopy type is the same that $E_\A$, see Proposition \ref{type}. 
 

\subsection{Cycles of the boundary manifold} \label{bound}
\vspace{0.25cm}

	We suppose that the admissible path $\gamma$ emanates from $x_0$ and  goes through $x_1,\dots,x_k$, the images of the singular points of $\A$ by $\pi$, ordered by their real parts. Let  $\mathscr{E}= \left\{\xi_{s,t}\right\}$ be the generating set of cycles of $\Gamma_\A$ defined in Section~\ref{BoundaryManifoldPresentation}.  
	
Each cycle $\xi_{s,t} \in \mathscr{E}$ is sent to $B_\A$ via $W_\A$, as follows. Let $X_0=(x_0,y_0)$ be a point of $\mathcal{N}_0$ such that $x_0=\gamma(0)$. The vertices of $\xi_{s,t} $ of the form $v_{L_0}$ or $v_{P}$ with $P \in L_0$ and the edges of the form $e(L_0,P)$, with $P\in L_0$, are all sent to the point $X_0$. The edges $e(L_i,P)$, with $i\neq 0$ and $P\in L_0$, are sent to segments from $X_0$ to the points $L_i\cap \pi\inv(x_0)$. Then the remaining vertices of the form $v_{L} \in V_L(\A) $ are sent to $L \cap \pi^{-1}(x_0)$. Let $\xi_{s,t}$ denote now the cycle of $W_{\A}$, relative to the left endpoints, where the vertices $v_P \in V_P(\A)$ are identified with the singular points $P$, and the edges with their corresponding wire of $W_{\A}$. 
 
 A \emph{framed cycle} is obtained by pushing a cycle $\xi_{s,t}$ to $B_\A$. This cycle $\xi_{s,t}$ consists of four arcs. Two of them are segments in $\mathcal{N}_0$, the two others are the parts in $L_s$ and $L_t$, see Figure~\ref{BoundaryManifoldGlobal}. The last two arcs go through several actual crossings of $W_\A$ and can be viewed as a union of small arcs. Each of them is projected to  $B_\A$  in the direction $[0:i:0]$ and their images are glued together as follows. For each actual crossing $P$, modify $\gamma$  slightly so that it makes a half circle of (small) radius $\eta_{P}$ around $x=\pi(P)$ in the positive sense. Choose $\eta_{P}$ so that the preimage of this half circle lies in $\mathcal{N}_i\cap\mathcal{S}_P\subset B_\A$ (also called $T_P$ or $T_{L_i}$ in Sub-section~\ref{BoundaryManifoldConstruction}), where $i\in\set{s,t}$. See Figure~\ref{BoundaryManifoldCross:A}. We avoid the intersection point $P:(x_P,y_P)$ of $L_s$ and $L_t$ as follows. Consider $\mathcal{S}_P$ as a polydisc, and join the two end points with the union of the two segments joining these end points with the point of $\pi\inv(x_P-\eta_P) \cap \mathcal{S}_P$ having the smallest real part, see Figure~\ref{BoundaryManifoldCross:B}.	
	
	The class of the obtained cycle in $\pi_1(B_\A,X_0)$, denoted by $\epsilon_{s,t}$, equals $\mathfrak{e}_{s,t}$ in the presentation of $\p(B_\A)$ found in Proposition~\ref{BoundaryPresentation}.

\begin{figure}[!ht]
	\subfigure[Near $P'\neq L_s\cap L_t$]{
		\begin{tikzpicture}
	\begin{scope}[shift={(-4,0)}, scale=0.6]
	
		\node (C1) at (2,2.5) {};
			\node (C11) at (2,1) {};
			\node (C12) at (2,4) {};
		\node (C2) at (8,2.5) {};
			\node (C21) at (8,1) {};
			\node (C22) at (8,4) {};
		\node (C3) at (5,0) {};
		\node (C4) at (5,5) {};	
		\node (Vp) at (1.375,2.5) {};
		\node (P11) at (1.88,1.2) {};
		\node (P12) at (1.88,3.8) {};		
		\node (P11) at (1.88,1.2) {};
		\node (P12) at (1.88,3.8) {};
		\node (P21) at (7.91,3.7) {};
		\node (P22) at (7.91,1.3) {};		
		\node (Pi1) at (0,3.5) {};
		\node (Pi2) at (0,1.5) {};					

		\draw (0,0) -- (2,1);
		\draw (0,5) -- (2,4);
		\draw (0,-0.5) -- (2,0.5);
		\draw (0,0.5) -- (2,1.5);
		\draw (0,4.5) -- (2,3.5);
		\draw (0,5.5) -- (2,4.5);
		\draw[thick,cap=round,line width=1.5pt] (0,4.75) -- (1.88,3.8);
		
		\draw[thick,color=white,line width=3pt,xscale=0.25] (C1) circle (2.5);
		\draw[white,fill=white] (1.9,0.9) circle (0.07);
		\draw[white,fill=white] (1.9,4.1) circle (0.07);
		\draw[thick,color=white,line width=3pt] (Vp.center) -- (P11);
		\draw[thick,color=white,line width=3pt] (Vp.center) -- (P12);			
		\draw[xscale=0.25] (C1) circle (2.5);
		
		\node at (7,5.7) {$\mathcal{S}_{P'}$};
		\draw[thick,color=white,line width=3pt] (2,1) -- (8,4);
		\draw[thick,color=white,line width=3pt] (2,4) -- (8,1);
		\draw (2,1) -- (8,4);
		\draw (2,4) -- (8,1);
		\draw[xscale=0.25] (C11) circle (0.5);						
		
		\draw[thick,color=white,line width=3pt] (P11) to[out=-43,in=-122] (P21);
		\draw (P11.center) to[out=-45,in=-120] (P21.center);
		
		\draw[thick,color=white,line width=3pt] (P12.center) to[out=-90,in=-140] (P22.center);
		\draw[thick,line width=1.5pt] (P12.center) to[out=-90,in=-140] (P22.center);
		
		\draw[xscale=0.25] (C12) circle (0.5);						
		
		\draw[thick,color=white,line width=3pt,xscale=0.25] (C2) circle (2.5);	
		\draw[white,fill=white] (7.88,1.05) circle (0.07);
		\draw[white,fill=white] (7.88,3.95) circle (0.07);
		\draw[xscale=0.25] (C2) circle (2.5);			
		\draw[xscale=0.25] (C21) circle (0.5);
		\draw[xscale=0.25] (C22) circle (0.5);
					
		\draw[thick,color=white,line width=3pt] (8,4) -- (10,5);														
		\draw[thick,color=white,line width=3pt] (8,1) -- (10,0);
		\node (A1) at (8,0.5) {};
		\node (A2) at (8,1.5) {};
		\node (A3) at (8,3.5) {};
		\node (A4) at (8,4.5) {};
		\draw[thick,color=white,line width=3pt] (A1) -- (10,-0.5);
		\draw[thick,color=white,line width=3pt] (A2) -- (10,0.5);
		\draw[thick,color=white,line width=3pt] (A3) -- (10,4.5);
		\draw[thick,color=white,line width=3pt] (A4) -- (10,5.5);			
		\draw[thick,color=white,line width=3pt,cap=round] (10,4.75) -- (P21);
		\draw[thick,color=white,line width=3pt,cap=round] (10, 0.25) -- (P22);						
		\draw (8,4) -- (10,5);														
		\draw (8,1) -- (10,0);
		\draw (8,0.5) -- (10,-0.5);
		\draw (8,1.5) -- (10,0.5);
		\draw (8,3.5) -- (10,4.5);
		\draw (8,4.5) -- (10,5.5);									
		\draw[cap=round] (10,4.75) -- (P21.center);
		\draw[cap=round,thick,line width=1.5pt] (10, 0.25) -- (P22.center);

		\draw[yscale=0.15] (C3) circle (3);
		\draw[yscale=0.15] (C4) circle (3);	
		
		\node at (P22) {$\bullet$};			
		\node at (P12) {$\bullet$};				
		\node (P) at (5,2.5) [above] {$P'$};			
		\node at (0,0) [left] {$L_t$};												
		\node at (0,5) [left] {$L_s$};	
									
	\end{scope}		
\end{tikzpicture}
		\label{BoundaryManifoldCross:A}
	}
	\subfigure[Near $P=L_s \cap L_t$]{
		\begin{tikzpicture}
	\begin{scope}[shift={(4,0)}, scale=0.6]
	
		\node (C1) at (2,2.5) {};
			\node (C11) at (2,1) {};
			\node (C12) at (2,4) {};
		\node (C2) at (8,2.5) {};
			\node (C21) at (8,1) {};
			\node (C22) at (8,4) {};
		\node (C3) at (5,0) {};
		\node (C4) at (5,5) {};	
		\node (Vp) at (1.375,2.5) {};
		\node (P11) at (1.88,1.2) {};
		\node (P12) at (1.88,3.8) {};		
		\node (P11) at (1.88,1.2) {};
		\node (P12) at (1.88,3.8) {};
		\node (P21) at (7.91,3.7) {};
		\node (P22) at (7.91,1.3) {};		
		\node (Pi1) at (0,3.5) {};
		\node (Pi2) at (0,1.5) {};					

		\draw (0,0) -- (2,1);
		\draw (0,5) -- (2,4);
		\draw (0,-0.5) -- (2,0.5);
		\draw (0,0.5) -- (2,1.5);
		\draw (0,4.5) -- (2,3.5);
		\draw (0,5.5) -- (2,4.5);
		\draw[thick,cap=round,line width=1.5pt] (0,0.25) -- (1.88,1.2);
		\draw[thick,cap=round,line width=1.5pt] (0,4.75) -- (1.88,3.8);
		
		\draw[thick,color=white,line width=3pt,xscale=0.25] (C1) circle (2.5);
		\draw[white,fill=white] (1.9,0.9) circle (0.07);
		\draw[white,fill=white] (1.9,4.1) circle (0.07);
		\draw[thick,color=white,line width=3pt] (Vp.center) -- (P11);
		\draw[thick,color=white,line width=3pt] (Vp.center) -- (P12);			
		\draw[xscale=0.25] (C1) circle (2.5);
		
		\node at (7,5.7) {$\mathcal{S}_{P}$};
		\draw[thick,color=white,line width=3pt] (2,1) -- (8,4);
		\draw[thick,color=white,line width=3pt] (2,4) -- (8,1);
		\draw (2,1) -- (8,4);
		\draw (2,4) -- (8,1);
		\draw[xscale=0.25] (C11) circle (0.5);						
		
		\draw[thick,color=white,line width=3pt] (P11) to[out=-43,in=-122] (P21);
		\draw (P11.center) to[out=-45,in=-120] (P21.center);
		
		\draw[thick,color=white,line width=3pt] (P12.center) to[out=-90,in=-140] (P22.center);
		\draw (P12.center) to[out=-90,in=-140] (P22.center);
		
		\draw[xscale=0.25] (C12) circle (0.5);						
		\draw[thick,line width=1.5pt] (Vp.center) -- (1.88,3.8);	
		\draw[thick,line width=1.5pt] (Vp.center) -- (1.88,1.2);
		
		\draw[thick,color=white,line width=3pt,xscale=0.25] (C2) circle (2.5);	
		\draw[white,fill=white] (7.88,1.05) circle (0.07);
		\draw[white,fill=white] (7.88,3.95) circle (0.07);
		\draw[xscale=0.25] (C2) circle (2.5);			
		\draw[xscale=0.25] (C21) circle (0.5);
		\draw[xscale=0.25] (C22) circle (0.5);
					
		\draw[thick,color=white,line width=3pt] (8,4) -- (10,5);														
		\draw[thick,color=white,line width=3pt] (8,1) -- (10,0);
		\node (A1) at (8,0.5) {};
		\node (A2) at (8,1.5) {};
		\node (A3) at (8,3.5) {};
		\node (A4) at (8,4.5) {};
		\draw[thick,color=white,line width=3pt] (A1) -- (10,-0.5);
		\draw[thick,color=white,line width=3pt] (A2) -- (10,0.5);
		\draw[thick,color=white,line width=3pt] (A3) -- (10,4.5);
		\draw[thick,color=white,line width=3pt] (A4) -- (10,5.5);			
		\draw[thick,color=white,line width=3pt,cap=round] (10,4.75) -- (P21);
		\draw[thick,color=white,line width=3pt,cap=round] (10, 0.25) -- (P22);						
		\draw (8,4) -- (10,5);														
		\draw (8,1) -- (10,0);
		\draw (8,0.5) -- (10,-0.5);
		\draw (8,1.5) -- (10,0.5);
		\draw (8,3.5) -- (10,4.5);
		\draw (8,4.5) -- (10,5.5);									
		\draw[cap=round] (10,4.75) -- (P21.center);
		\draw[cap=round] (10, 0.25) -- (P22.center);

		\draw[yscale=0.15] (C3) circle (3);
		\draw[yscale=0.15] (C4) circle (3);	
		
		\node at (P12) {$\bullet$};				
		\node at (P11) {$\bullet$};
		\node at (Vp) {$\bullet$};	
		\node (P) at (5,2.5) [above] {$P_{s,t}$};			
		\node at (0,0) [left] {$L_t$};												
		\node at (0,5) [left] {$L_s$};	
									
	\end{scope}		
\end{tikzpicture}
		\label{BoundaryManifoldCross:B}
	}
	\caption{Construction of $\epsilon_{s,t}$ near singular points	}
	\label{BoundaryManifoldCross}
\end{figure}
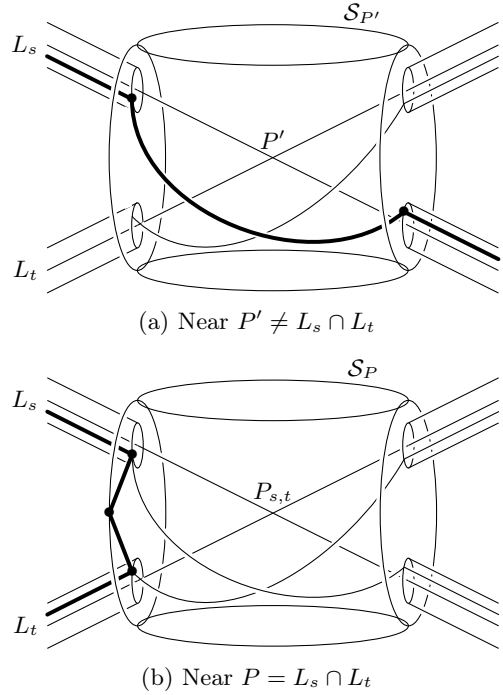

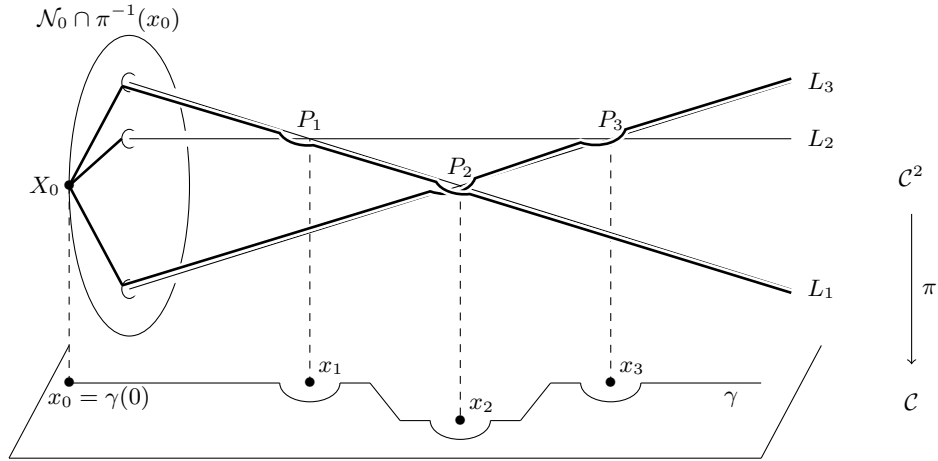
\begin{figure}[!ht]
	\begin{tikzpicture}		
	\begin{scope}[yscale=0.5,xscale=0.8]
		\draw (0,0)--(12.5,0);
		\draw (0,0)--(1,3);
		\draw (12.5,0)--(13.5,3);
		
		\node (Q1) at (5,2) {$\bullet$};
			\node[above right] at (Q1) {$x_1$};
		\node (Q2) at (7.5,1) {$\bullet$};
			\node[above right] at (Q2) {$x_2$};
		\node (Q3) at (10,2) {$\bullet$};
			\node[above right] at (Q3) {$x_3$};
		
		\node (x0) at (1,2) {$\bullet$};
			\node at (1.5,1.6) {$x_0=\gamma(0)$};
		\draw (1,2)--(4.5,2);
		\draw (4.5,2) arc (180:360:0.5) ;
		\draw (5.5,2)--(6,2);
		\draw (6,2)--(6.5,1);
		\draw (6.5,1)--(7,1);
		\draw (7,1) arc (180:360:0.5) ;
		\draw (8,1)--(8.5,1);
		\draw (8.5,1)--(9,2);
		\draw (9,2)--(9.5,2);
		\draw (9.5,2) arc (180:360:0.5) ;
		\draw (10.5,2)--(12.5,2);
		\node[below] at (12,2) {$\gamma$};
		
		\node at (3,11.7) [left] {$\mathcal{N}_0\cap \pi\inv(x_0)$};
		\node (X0) at (1,7.25) {$\bullet$};
			\node[left] at (X0) {$X_0$};
		\node (C) at (2,7.25) {};	
		\node (C1) at (2,10) {};	
		\node (C2) at (2,8.5) {};	
		\node (C3) at (2,4.5) {};		
		\draw[xscale=0.25] (C) circle (4);			
		\draw[xscale=0.25,yscale=0.5] (C1) circle (0.5);		
		\draw[xscale=0.25,yscale=0.5] (C2) circle (0.5);
		\draw[xscale=0.25,yscale=0.5] (C3) circle (0.5);
					
		\node (G1) at (1.885,9.9) {};
		\node (G2) at (1.875,8.5) {};
		\node (G3) at (1.885,4.6) {};
		\draw[thick,line width=1pt] (X0.center)--(G1.center) 
		;
		\draw[thick,line width=1pt] (X0.center)--(G2.center) 
		;
		\draw[thick,line width=1pt] (X0.center)--(G3.center) 
		;
		
		\draw[thick,color=white,line width=6pt] (2,10)--(13,4.5);
		\draw[thick,color=white,line width=6pt] (2,8.5)--(13,8.5);
		\draw[thick,color=white,line width=6pt] (2,4.5)--(13,10);
		
		\draw (2,10)--(13,4.5);
			\node at (13.5,4.5) {$L_1$};
		\draw (2,8.5)--(13,8.5);
			\node at (13.5,8.5) {$L_2$};
		\draw (2,4.5)--(13,10);
			\node at (13.5,10) {$L_3$};
		\node (P1) at (5,8.5) [above] {$P_1$};
		\node (P2) at (7.5,7.25) [above] {$P_2$};
		\node (P3) at (10,8.5) [above] {$P_3$};
					
		\node at (15,7.5) {$\C^2$};
		\draw[->] (15,6.5)--(15,2.5);		
		\node at (15,1.5) {$\C$};	
		\node at (15.3,4.5) {$\pi$};
		
		\draw[dashed] (P1)--(5,2);
		\draw[dashed] (P2)--(7.5,1);
		\draw[dashed] (P3)--(10,2);
		\draw[dashed] (X0.center)--(x0.center);			
		
		\draw[thick,color=white,line width=3pt] (G3) -- (7,7.1) arc (240:350:0.5) -- (9.5,8.4) arc (240:350:0.5) -- (13,10.1);
		\draw[thick,line width=1pt] (G3.center) -- (7,7.1) arc (240:350:0.5) -- (9.5,8.4) arc (240:350:0.5) -- (13,10.1) 
		;	
		\draw[thick,color=white,line width=3pt] (G1) -- (4.5,8.65) arc (200:280:0.5) -- (7.1,7.35) arc (200:280:0.5) -- (13,4.4);
		\draw[thick,line width=1pt] (G1.center) -- (4.5,8.65) arc (200:280:0.5) -- (7.1,7.35) arc (200:280:0.5) -- (13,4.4) 
		;		
		
	
	\end{scope}
\end{tikzpicture}	
	\caption{Construction of $\delta(\epsilon)$ }
	\label{BoundaryManifoldGlobal} 
\end{figure}

In order to compute the images of the framed cycles in the complement $E_\A$, we introduce \emph{geometric cycles} $\mathcal{E}_{s,t}$, defined as parallel copies of the  $\xi_{s,t}$'s. Indeed, let $\mathcal{E}_{s,t}$ in $\pi_1(B_\A)$ be the image of $\xi_{s,t}$ by the projection in the direction $[0:i:0]$. Remark that the difference between $\epsilon_{s,t}$ and $\mathcal{E}_{s,t}$ is local and takes place near the singular points. We also define the \emph{unknotting map} by:
\begin{equation*}	
	\delta :  \left\{
		\begin{array}{ccc}
			\p(B_\A, X_0) & \longrightarrow & \p(B_\A, X_0) \\
			\alpha_i & \longmapsto & \alpha_i \\
			\epsilon_{s,t} & \longmapsto & \mathcal{E}_{s,t}
		\end{array}
		\right..
\end{equation*}

Let us define $\delta_{s,t}^l$ (resp. $\delta_{s,t}^r$)  as the products over all actual crossings $P$ of the arc $L_s$ (resp. $L_t$) of $\xi_{s,t}$, different from $L_s \cap L_t$, of the following words:

Suppose that $P= L_{i_1} \cap \dots \cap L_{i_m}$ where the order of the lines corresponds to Figure \ref{Delta}.
\begin{itemize}
	\item[-] If $P\in L_s$, let $h\in\left\{1,\cdots,m\right\}$ be such that $i_h=s$, then $P$ contributes to $\delta^l_{s,t}$ by: 
		\begin{equation*}
			\epsilon_{i_1,i_h}\inv \left( \alpha_{i_1}\inv \left( \epsilon_{i_1,i_{2}} \alpha_{i_{2}}\inv \epsilon_{i_1,i_{2}}\inv \right) \cdots \left( \epsilon_{i_1,i_{h-1}} \alpha_{i_{h-1}}\inv \epsilon_{i_1,i_{h-1}}\inv \right) \right) \epsilon_{i_1,i_h},
		\end{equation*}
	\item[-] If $P\in L_t$, let $h\in\left\{1,\cdots,m\right\}$ be such that $i_h=t$, then $P$ contributes to $\delta^r_{s,t}$ by: 
		\begin{equation*}
			\epsilon_{i_1,i_h}\inv \left(\left( \epsilon_{i_1,i_{h-1}} \alpha_{i_{h-1}} \epsilon_{i_1,i_{h-1}}\inv \right) \cdots \left( \epsilon_{i_1,i_{2}} \alpha_{i_{2}} \epsilon_{i_1,i_{2}}\inv \right) \alpha_{i_1} \right) \epsilon_{i_1,i_h}.
		\end{equation*}
\end{itemize}
	 
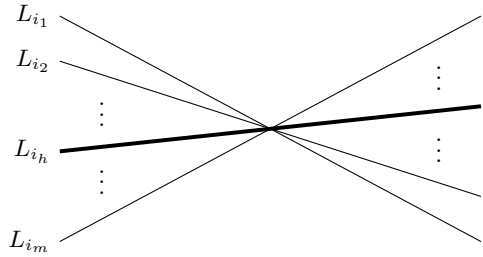
\begin{figure}[!ht]
	\begin{tikzpicture}
	\begin{scope}[shift={(0,0)},xscale=2.8,yscale=0.6]
		\draw (-1,3)--(1,-2);
		\draw (-1,2)--(1,-1);
		\draw[thick,line width=1.5pt] (-1,0)--(1,1);
		\draw (-1,-2)--(1,3);
		\node at (-0.8,1) {$\vdots$};
		\node at (-0.8,-0.5) {$\vdots$};
		\node at (0.8,1.8) {$\vdots$};
		\node at (0.8,0.2) {$\vdots$};
		\node[left] at (-1,3) {$L_{i_1}$};
		\node[left] at (-1,2) {$L_{i_2}$};
		\node[left] at (-1,0) {$L_{i_h}$};
		\node[left] at (-1,-2) {$L_{i_m}$};
	\end{scope}
\end{tikzpicture}
	\caption{Indexation of a crossing}
	\label{Delta}
\end{figure}

\begin{propo} \label{image}
	The image of a framed cycle by the unknotting map $\delta$ is:
	\begin{equation*}
	 \delta(\epsilon_{s,t})  = \mathcal{E}_{s,t}= \delta_{s,t}^l \ \epsilon_{s,t} \ \delta_{s,t}^r.
	\end{equation*}
\end{propo}

\begin{proof*}
The contribution of $P$ is induced by the action of a half-twist, given by the pre-image by $\gamma$ of the half circle 
around each $x=\pi(P)$, in the positive sense. We obtain the description of $\delta_{s,t}^l$ and $\delta_{s,t}^r$ above, and then $\epsilon_{s,t}= \left(\delta_{s,t}^l\right)^{-1} \mathcal{E}_{s,t} \left(\delta_{s,t}^r\right)^{-1}$. 
\end{proof*}

\begin{ex}
	The images of the $\epsilon_{s,t}$ of the didactic example by the unknotting map are:
	\begin{equation*}
		\begin{array}{c}
			\delta(\epsilon_{1,2})=\epsilon_{1,2}, \qquad
			\delta(\epsilon_{1,4})=\epsilon_{1,4}, \qquad
			\delta(\epsilon_{1,3})=\epsilon_{1,3},\\[7pt]
			\delta(\epsilon_{2,3}) = (\epsilon_{1,2}\inv\alpha_1\inv\epsilon_{1,2}) \epsilon_{2,3} (\epsilon_{1,3}\inv\alpha_1\epsilon_{1,3}), \\[7pt]
			\delta(\epsilon_{3,4}) =  (\epsilon_{1,3}\inv\alpha_1\inv\epsilon_{1,3}) \epsilon_{3,4} ( \epsilon_{1,4}\inv(\alpha_1\epsilon_{1,2}\alpha_2\epsilon_{1,2}\inv)\epsilon_{1,4}). \\
		\end{array}
	\end{equation*}
\end{ex}

\subsection{Inclusion map} \label{inclusion}

Geometric cycles were constructed by taking parallel copies of cycles of $\Gamma_{\A}$, via $W_{\A}$, to the boundary manifold $B_\A$.
 Their image in $E_\A$ can then be computed directly from $W_\A$.

Let $\xi_{s,t}$ be a cycle of $W_{\A}$, relative to the left endpoints. An \emph{over arc} $\varsigma$ is an arc of $W_\A$ that goes over $\xi_{s,t}$ through a virtual crossing.  Denote $\sgn(\varsigma)\in\left\{\pm 1 \right\}$ the sign of the crossing. It is positive if the orientations of $\varsigma$ and $\xi_{s,t}$ (in this order) at the crossing form a positive base, and is negative otherwise.

	Let $S_{\xi_{s,t}}$ be the set of over arcs of $\xi_{s,t}$ --oriented from left to right--. The element $\mu_{s,t}$ is defined by:
\begin{equation*}
	\mu_{s,t}=\prod\limits_{\varsigma\in S_{\xi_{s,t}}} a_\varsigma^{\sgn(\varsigma)}, 
\end{equation*}
	where $a_\varsigma$ is the word associate to the arc $\varsigma$ by the Arvola's algorithm (see Subsection~\ref{color}), and the order in the product respects the order of the virtual crossings in the cycle $\xi_{s,t}$. Note that $\mu_{s,t}$ is a product of conjugates of meridians.

\begin{thm}\label{MainResult}
	For $i=0,\dots,n$, let $\alpha_i$ be the meridians of the lines and let $\left\{\epsilon_{s,t}\right\}$ be a set of cycles indexed by a generating system $\mathscr{E}$ of cycles of  the incidence graph $\Gamma_{\A}$. Then the fundamental group of $B_\A$ is generated by $\left\{ \alpha_0, \dots, \alpha_n, \epsilon_{s_1,t_1},\dots,\epsilon_{s_l,t_l} \right\}$, and the map $i_*:\p(B_\A)\rightarrow \p(E_\A)$ induced by the inclusion is described as follows :
	\begin{equation*} i_* : 
	\left\{
		\begin{array}{ccl}
			\alpha_i & \longmapsto & \alpha_i, \\			
			\epsilon_{s,t} & \longmapsto & \left(\delta^l_{s,t}\right)\inv \mu_{s,t} \left(\delta^r_{s,t}\right)\inv,
		\end{array}
	\right. 
	\end{equation*}
\end{thm}

It is worth noticing that using a recursive argument on the set of $\epsilon_{s,t}$, the words $\left(\delta^l_{s,t}\right)\inv \mu_{s,t} \left(\delta^r_{s,t}\right)\inv$ are products of conjugates of the meridians $\alpha_1, \dots, \alpha_n$.

\begin{proof*}
Since $B_\A\subset E_\A$, then a class in $\p(B_\A)$ can be viewed as a class in $\p(E_\A)$, and both are denoted in the same way. 

By Proposition \ref{image}, each class $\delta_{s,t}^l \epsilon_{s,t} \delta_{s,t}^r$ in $\p(B_\A)$ can be represented by a geometric cycle $\mathcal{E}_{s,t}$, obtained as a parallel copy of $\xi_{s,t}$ from $W_\A$ to $B_\A$. Consider a $2$-cell homotopic to a disc with $\hbox{card}(S_{\xi_{s,t}})$ holes. Then glue the boundary of the disc to $\mathcal{E}_{s,t}$ and the other boundary components to the meridians of over arcs $\varsigma\in S_{\xi_{s,t}} $. As the 2-cell is in $E_\A$, then, in the exterior, $\mathcal{E}_{s,t}$ can be retracted to the product $\mu_{s,t}$ of the $a_\varsigma$, with $\varsigma\in S_{\xi_{s,t}}$. It follows that in $\p(E_\A)$, $\delta_{s,t}^l \epsilon_{s,t} \delta_{s,t}^r=\mu_{s,t}$. 
\end{proof*} 

In the construction of the boundary manifold (see Sub-section~\ref{BoundaryManifoldConstruction}), the cycle $\mathfrak{e}_{s,t}$ appearing as cycles in HNN-extension. In the previous theorem, we take $\mathfrak{e}_{s,t}=\epsilon_{s,t}$ as cycles in the generating system of $\p(B_\A)$. Now we consider the geometric cycles defined from the projection of $W_\A$ in $B_\A$ (i.e. the $\delta(\epsilon_{s,t})$, noted $\mathcal{E}_{s,t}$). Their construction allows to consider them as cycles for the HNN-extension too (i.e. $\mathfrak{e}_{s,t}=\mathcal{E}_{s,t}$). With this generating system of $\p(B_\A)$, we obtain a simpler version of Theorem~\ref{MainResult}:

\begin{thm}\label{Simpler_MainResult}
	For $i=0,\dots,n$, let $\alpha_i$ be the meridians of the lines and let $\left\{\mathcal{E}_{s,t}\right\}$  be a set of cycles indexed by a generating system $\mathscr{E}$ of cycles of  the incidence graph $\Gamma_{\A}$. Then the fundamental group of $B_\A$ is generated by $\left\{ \alpha_1, \dots, \alpha_n, \mathcal{E}_{s_1,t_1},\cdots,\mathcal{E}_{s_l,t_l} \right\}$, and the map $i_*:\p(B_\A)\rightarrow \p(E_\A)$ induced by the inclusion is described as follows :
	\begin{equation*} i_* : 
	\left\{
		\begin{array}{ccl}
			\alpha_i & \longmapsto & \alpha_i, \\			
			\mathcal{E}_{s,t} & \longmapsto & \mu_{s,t},
		\end{array}
	\right. 
	\end{equation*}
\end{thm}

\subsection{Exact sequence} \label{sequ}

\begin{thm} \label{sequence}
The following sequence is exact
\begin{equation*}
	0 \longrightarrow K \stackrel{\phi}{\longrightarrow} \p(B_\A) \stackrel{i_*}{\longrightarrow} \p(E_\A) \longrightarrow 0,
\end{equation*}
where $K$ is the normal subgroup of $\p(B_\A)$ generated by all the elements of the form $\delta_{s,t}^l \epsilon_{s,t} \delta_{s,t}^r \mu_{s,t}\inv$, and the product $\alpha_0\cdots\alpha_n$.
\end{thm}

\begin{proof*}
By Theorem \ref{MainResult}, the map $i_*$ is onto and $K$ is included in $\ker(i_*)$.
It remains to show that the relations induced by the images $i_*(\epsilon_{s,t})$  are enough to determine
a presentation of $\pi_1(E_\A)$. We compare these relations to those coming from braid monodromy 
 and Zariski-Van Kampen's method, see~\cite{li:86} for example.
 
Let $P=L_{i_1}\cap\cdots\cap L_{i_m}$ (as in Figure~\ref{Delta}), be a singular point of $\A$, with $i_1=\nu(P)$. Consider a small ball in $\mathbb{P}^2$ with center $P$ and a local base point $b$ in its boundary
 sphere. Let $\lambda$ be a path from $X_0$ to $b$, and let $y_j$ be the (local) meridian of $L_j$ with base $b$, for
  $j=1,\dots,m$.  The path $\lambda$ can be chosen in such a way that Zariski-Van Kampen's relations associated to $P$ are :
\begin{equation*}
	[y_{i,i_m}^\lambda,\cdots,y_{i,i_1}^\lambda].
\end{equation*}
We can assume that $b$ is a point of $\epsilon_{i_1,j}$, for all $j= i_2,\cdots,i_m$. Then $\epsilon_{i_1,j}=\beta_j\inv \beta_{i_1}$ where $\beta_{i_1}$ goes from $X_0$ to $b$, and  $\beta_{j}\inv$ from $b$ to $X_0$. We get
\begin{equation*}
	\begin{array}{rcl}
		[\alpha_{i_m}^{\epsilon_{i_1,i_m}},\cdots,\alpha_{i_2}^{\epsilon_{i_1,i_2}},\alpha_{i_1}]
		& \Leftrightarrow &	[\alpha_{i_m}^{\beta_{i_m}\inv\beta_{i_1}},\cdots,\alpha_{i_2}^{\beta_{i_2}\inv\beta_{i_1}},\alpha_{i_1}^{\beta_{i_1}\inv\beta_{i_1}}], \\[0.3cm]	
		& \Leftrightarrow &	[\alpha_{i_m}^{\beta_{i_m}\inv},\cdots,\alpha_{i_2}^{\beta_{i_2}\inv},\alpha_{i_1}^{\beta_{i_1}\inv}]^{\beta_{i_1}}, \\[0.3cm]	
		& \Leftrightarrow &	[\alpha_{i_m}^{\beta_{i_m}\inv},\cdots,\alpha_{i_2}^{\beta_{i_2}\inv},\alpha_{i_1}^{\beta_{i_1}\inv}], \\[0.3cm]
		& \Leftrightarrow &	[\alpha_{i_m}^{\beta_{i_m}\inv},\cdots,\alpha_{i_2}^{\beta_{i_2}\inv},\alpha_{i_1}^{\beta_{i_1}\inv}]^\lambda. \\	
	\end{array}
\end{equation*}

Note that during this computation, the base point may have changed, but the first and the last relations are based in $X_0$.
Since $\alpha_{j}^{\beta_j\inv}=y_{i,j}$, for all $j=i_1,\cdots,i_m$, then:

\begin{equation*}
	\begin{array}{rcl}
		[\alpha_{i_m}^{\epsilon_{i_1,i_m}},\cdots,\alpha_{i_2}^{\epsilon_{i_1,i_2}},\alpha_{i_1}] & \Leftrightarrow & [y_{i_m},\cdots,y_{i_2},y_{i_1}]^\lambda, \\[0.3cm]
		& \Leftrightarrow & [y_{i_m}^\lambda,\cdots,y_{i_2}^\lambda,y_{i_1}^\lambda]. \\
	\end{array}
\end{equation*}
\end{proof*} 


\subsection{Homotopy type of the complement} \label{homotopy}

From Theorem~\ref{sequ}, we obtain a presentation of the fundamental group of $\p(E_\A)$.

\begin{cor}\label{MainCor}
	For $i=1, \dots, n$, let $\alpha_i$ be the meridians of the lines $L_i$. For any singular point $P=L_{i_1}\cap L_{i_2} \cap \cdots \cap L_{i_m}$ with $i_1=\nu(P)$, let 
	\begin{equation*}
		\mathcal{R}_P = [\alpha_{i_m}^{c_{i_m}},\cdots,\alpha_{i_2}^{c_{i_2}},\alpha_{i_1}], \hbox{ where } c_{i_j} = \left(\delta^l_{i_1,i_j}\right)\inv \mu_{i_1,i_j} \left(\delta^r_{i_1,i_j}\right)\inv \hbox{ for all } j=2,\cdots,m.
	\end{equation*}
	The fundamental group of  $E_\A$ admits the following presentation:
	\begin{equation*}
	 \p(E_\A)=\langle \alpha_1,\cdots,\alpha_n\ \mid\ \bigcup\limits_{P\in\P} \mathcal{R}_P \rangle .
	\end{equation*}
\end{cor}

\begin{rmk}
	This corollary can not be simplify using geometric cycles, since relations in the presentation of $\p(B_\A)$ --given in Proposition~\ref{BoundaryPresentation}-- are wrong with these cycles.
\end{rmk}

\begin{proof*}
	For each $\epsilon_{s,t}$, let $r_{s,t}$ be the relation $\epsilon_{s,t} =\left(\delta^l_{s,t}\right)\inv \mu_{s,t} \left( \delta^r_{s,t} \right)\inv$, and for each point $P\in\P$ (with $P=L_{i_1}\cap\cdots\cap L_{i_m}$ and $i_1=\nu(P)$), we define the relation
	$ \mathcal{R}'_P:\  [\alpha_{i_m}^{\epsilon_{i_1,i_m}},\cdots,\alpha_{i_2}^{\epsilon_{i_1,i_2}},\alpha_{i_1}]$.
	Then, Theorem \ref{sequence} implies that we have the following presentation:
	\begin{equation*}
		\p(E_\A)\ =\ \langle \alpha_0,\alpha_1,\cdots,\alpha_n,\epsilon_{s_1,t_1},\cdots,\epsilon_{s_l,t_l} \ |\  \bigcup\limits_{P\in\P} \mathcal{R}'_P,\ \bigcup\limits_{i=1}\limits^l r_{s_i,t_i},\ \alpha_0\cdots\alpha_n \rangle .
	\end{equation*}
	Consider the total order on the set $\left\{\epsilon_{s,t}\right\}$: $\left( \epsilon_{s,t}<\epsilon_{s',t'} \right) \Leftrightarrow \left(s \leq s'\ and\ t<t' \right)$. By construction, $\delta_{s,t}^l$ and $\delta_{s,t}^r$ depend on $\epsilon_{s',t'}$ if and only if $\epsilon_{s',t'}<\epsilon_{s,t}$. Since $\mu_{s,t}$ is a product of meridians, then the smallest $\epsilon_{s,t}$ is a product of meridians. And by induction, the relation $r_{s,t}$ expresses any $\epsilon_{s,t}$ as a product of $\alpha_i$. 
	
	Finally, using the relation $\alpha_0,\cdots,\alpha_n=1$, the meridian $\alpha_0$ can be removed from the set of generators of $\p(E_\A)$. Indeed no other relation contains $\alpha_0$.
\end{proof*}

\begin{ex}
	The presentation of the fundamental group of the didactic example is:
	\begin{equation*}
		<\alpha_1,\alpha_2,\alpha_3,\alpha_4\ \mid\ 
		[\alpha_4^{\alpha_3\inv},\alpha_2,\alpha_1],\ 
		[\alpha_3,\alpha_1],\ 
		[\alpha_3^{\alpha_1\alpha_2\inv\alpha_1\inv},\alpha_4] ,\ 
		[\alpha_3^{\alpha_1\alpha_4\inv},\alpha_2]>
	\end{equation*}
\end{ex}

\begin{propo} \label{type}
	The $2$-complex modeled on the minimal presentation given in Corollary \ref{MainCor}
is homotopy equivalent to $E_\A$.
\end{propo}

\begin{proof*}
	The proof of Theorem \ref{sequence} shows in particular that the relations of the presentation in Corollary \ref{MainCor} are
equivalent to Zariski-Van Kampen's relations, based on the braid monodromy. It is shown in \cite{li:86}  that the $2$-complex modeled on a minimal presentation equivalent to Zariski-Van Kampen's presentation is homotopy equivalent to $E_\A$.
\end{proof*}


\section{The example of positive MacLane line arrangement}\label{Example}

In this section, we illustrate Theorem \ref{MainResult} with an arrangement $Q^+$ introduced by S.~MacLane, given by the following equations
\begin{equation*}
	\begin{array}{ccc}
		L_0=\left\{ z=0 \right\}; &
		L_1=\left\{ z-x=0 \right\}; & 
		L_2=\left\{ x=0 \right\}; \\ 
		L_3=\left\{ y=0 \right\}; &
		L_4=\left\{ z+\omega^2 x+\omega y =0 \right\}; &
		L_5=\left\{ y-x=0 \right\}; \\ 
		L_6=\left\{ z-x-\omega^2 y =0 \right\}; & 
		L_7=\left\{ z+\omega y =0 \right\}, 
	\end{array}
	\end{equation*}
 	where $\omega=\exp(\frac{ 2 i\pi}{3})$ is a primitive root of unity of order $3$.
	
The incidence graph $\Gamma$ of $Q^+$ is given in Figure \ref{MacLaneIncidence}. 
\begin{figure}[!ht]	
	\begin{center}
	\begin{tikzpicture}[xscale=1.2,yscale=0.6]
	
		\node[draw] (L0) at (-1,5) {$L_0$};
		
		\node[draw] (P012) at (1,7.5) {$P_{0,1,2}$};
		\node[draw] (P034) at (1,5) {$P_{0,3,4}$};
		\node[draw] (P056) at (1,2.5) {$P_{0,5,6}$};
		\node[draw] (P07) at (1,0) {$P_{0,7}$};

		\node[draw] (L7) at (2.5,0) {$L_7$};
		\node[draw] (L6) at (2.5,1.33) {$L_6$};
		\node[draw] (L5) at (2.5,2.66) {$L_5$};
		\node[draw] (L4) at (2.5,4) {$L_4$};
		\node[draw] (L3) at (2.5,5.33) {$L_3$};
		\node[draw] (L2) at (2.5,6.66) {$L_2$};
		\node[draw] (L1) at (2.5,8) {$L_1$};
		
		\node[draw] (P157) at (6,8.4) {$P_{1,5,7}$};
		\node[draw] (P13) at (6,7.2) {$P_{1,3}$};
		\node[draw] (P146) at (6,6) {$P_{1,4,6}$};		
		\node[draw] (P235) at (6,4.8) {$P_{2,3,5}$};		
		\node[draw] (P247) at (6,3.6) {$P_{2,4,7}$};
		\node[draw] (P26) at (6,2.4) {$P_{2,6}$};
		\node[draw] (P367) at (6,1.2) {$P_{3,6,7}$};		
		\node[draw] (P45) at (6,0) {$P_{4,5}$};
		
		\draw[thick, line width=1.5pt, cap=round] (L0) -- (P012);
		\draw[thick, line width=1.5pt, cap=round] (L0) -- (P034);
		\draw[thick, line width=1.5pt, cap=round] (L0) -- (P056);
		\draw[thick, line width=1.5pt, cap=round] (L0) -- (P07);
		
		\draw[thick, line width=1.5pt, cap=round] (P012) -- (L1);
		\draw[thick, line width=1.5pt, cap=round] (P012) -- (L2);
		\draw[thick, line width=1.5pt, cap=round] (P034) -- (L3);
		\draw[thick, line width=1.5pt, cap=round] (P034) -- (L4);
		\draw[thick, line width=1.5pt, cap=round] (P056) -- (L5);
		\draw[thick, line width=1.5pt, cap=round] (P056) -- (L6);
		\draw[thick, line width=1.5pt, cap=round] (P07) -- (L7);
		
		\draw[thick, line width=1.5pt, cap=round] (L1) -- (P157);
		\draw[thick, line width=1.5pt, cap=round] (L1) -- (P13);
		\draw[thick, line width=1.5pt, cap=round] (L1) -- (P146);
		
		\draw[thick, line width=1.5pt, cap=round] (L2) -- (P235);
		\draw[thick, line width=1.5pt, cap=round] (L2) -- (P247);
		\draw[thick, line width=1.5pt, cap=round] (L2) -- (P26);
		
		\draw[dashed] (L3) -- (P235);
		\draw[thick, line width=1.5pt, cap=round] (L3) -- (P367);
		\draw[dashed] (L3) -- (P13);
		
		\draw[dashed] (L4) -- (P247);
		\draw[thick, line width=1.5pt, cap=round] (L4) -- (P45);
		\draw[dashed] (L4) -- (P146);
		
		\draw[dashed] (L5) -- (P235);
		\draw[dashed] (L5) -- (P45);
		\draw[dashed] (L5) -- (P157);
		
		\draw[dashed] (L6) -- (P26);
		\draw[dashed] (L6) -- (P367);
		\draw[dashed] (L6) -- (P146);
		
		\draw[dashed] (L7) -- (P247);
		\draw[dashed] (L7) -- (P367);
		\draw[dashed] (L7) -- (P157);
		
	\end{tikzpicture}	
		
\end{center}
	\caption{Incidence graph of MacLane's arrangement $Q^+$}
	\label{MacLaneIncidence}
\end{figure}
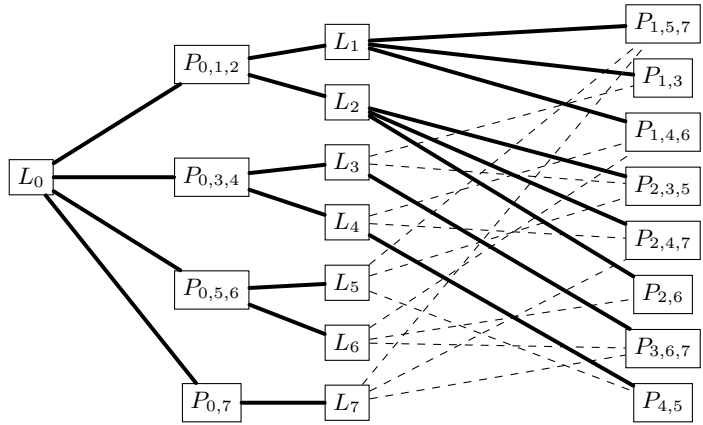

It is worth mentioning that $Q^+$ is one of the only two topological
 realisations of this combinatorial data by an arrangement in $\PP^2$. The other realisation $Q^-$ corresponds to $\omega=\exp(\frac{- 2 i\pi}{3})$.
 These two arrangements do not admit real equations.


\subsection*{Generating set of cycles of $\Gamma_{Q^+}$}
\vspace{0.25cm}

 Consider the maximal tree  $\T$ in $\Gamma_{Q^+}$ indicated with thick lines in Figure \ref{MacLaneIncidence}. Let $\mathscr{E}$ be the generating system of cycles induced by  $\T$ (it is in one-to-one correspondance with the dotted lines in Figure \ref{MacLaneIncidence}):
	\begin{equation*}
		\mathscr{E}=\left\{ \xi_{2,3}, \xi_{2,5}, \xi_{2,4}, \xi_{2,7}, \xi_{2,6}, \xi_{4,5}, \xi_{3,6}, \xi_{3,7}, \xi_{1,5}, \xi_{1,7}, \xi_{1,3}, \xi_{1,4}, \xi_{1,6} \right\}.
	\end{equation*}
	

\subsection*{Group of the boundary manifold}
\vspace{0.25cm}

By Section \ref{bound},  the images $\epsilon_{s,t}$ of the cycles $\xi_{s,t}$ in $B_{Q^+}$ form a family of cycles in $\p(B_{Q^+})$.
Proposition~\ref{BoundaryPresentation} applies to this explicit family, and $\pi_1(B_{Q^+})$ admits a presentation with generators:
\begin{equation*} 
	\left\{ 
	\alpha_0, \alpha_1, \alpha_2, \alpha_3, \alpha_4, \alpha_5, \alpha_6, \alpha_7,
	\mathfrak{e}_{2,3}, \mathfrak{e}_{2,5}, \mathfrak{e}_{2,4}, \mathfrak{e}_{2,7}, \mathfrak{e}_{2,6}, \mathfrak{e}_{4,5}, \mathfrak{e}_{3,6}, \mathfrak{e}_{3,7}, \mathfrak{e}_{1,5}, \mathfrak{e}_{1,7}, \mathfrak{e}_{1,3}, \mathfrak{e}_{1,4}, \mathfrak{e}_{1,6} \right\},
\end{equation*}
	and relations:
	\begin{multline*}
	[\alpha_7^{\mathfrak{e}_{1,7}},\alpha_5^{\mathfrak{e}_{1,5}},\alpha_1],
	[\alpha_3^{\mathfrak{e}_{1,3}},\alpha_1],
	[\alpha_6^{\mathfrak{e}_{1,6}},\alpha_4^{\mathfrak{e}_{1,4}},\alpha_1],
	[\alpha_5^{\mathfrak{e}_{2,5}},\alpha_3^{\mathfrak{e}_{2,3}},\alpha_2],\\
	[\alpha_7^{\mathfrak{e}_{2,7}},\alpha_4^{\mathfrak{e}_{2,4}},\alpha_2],
	[\alpha_6^{\mathfrak{e}_{2,6}},\alpha_2],
	[\alpha_7^{\mathfrak{e}_{3,7}},\alpha_6^{\mathfrak{e}_{3,6}},\alpha_3],
	[\alpha_5^{\mathfrak{e}_{4,5}},\alpha_4].
	\end{multline*}


\subsection*{Geometric cycles and unknotting map}
\vspace{0.25cm}

\begin{figure}[!ht]
	\begin{tikzpicture}
	\begin{scope}[shift={(0,0)},x=15pt,y=13pt]

		\node at (0,8) [below,left] {1};
		\node at (0,7) [below,left] {2};
		\node at (0,6) [below,left] {3};
		\node at (0,5) [below,left] {4};
		\node at (0,4) [below,left] {5};
		\node at (0,3) [below,left] {6};		
		\node at (0,2) [below,left] {7};
	
		\draw[thick,cap=round,color=white,line width=6pt] (0,4) -- (1,4) -- (2,5) -- (3,5) -- (4,7) --	(6,7) -- (9,4) -- (10,4) -- (10.5,4.5);
		\draw[thick,cap=round,line width=2pt] (0,4) -- (1,4) -- (2,5) -- (3,5) -- (4,7) --	(6,7) -- (9,4) -- (10,4) -- (10.5,4.5);			
		\draw[thick,cap=round,color=white,line width=6pt] (0,5) -- (1,5) -- (2,4) -- (8,4) -- (9,5) -- (10,5);
		\draw[thick,cap=round,line width=2pt] (0,5) -- (1,5) -- (2,4) -- (8,4) -- (9,5) -- (10,5) -- (10.5,4.5);	
				\ncross{8}{0};\ncross{7}{0};\ncross{6}{0};\ncross{5}{0};\ncross{4}{0};\ncross{3}{0};\ncross{2}{0};
				\ncross{8}{1};\ncross{7}{1};\ncross{6}{1};\ocross{4}{1};\ncross{3}{1};\ncross{2}{1};		
				\ncross{8}{2};\ncross{7}{2};\ncross{6}{2};\ncross{5}{2};\ncross{4}{2};\ucross{2}{2};
				\ncross{8}{3};\mcross{5}{3}{3};\ncross{4}{3};\ncross{3}{3};\ncross{2}{3};
				\ncross{8}{4};\ncross{7}{4};\ncross{6}{4};\mcross{3}{3}{4};\ncross{2}{4};
				\ncross{8}{5};\ncross{7}{5};\ncross{6}{5};\ncross{5}{5};\ncross{4}{5};\mcross{2}{2}{5};
				\ncross{8}{6};\ocross{6}{6};\ncross{5}{6};\ncross{4}{6};\ncross{3}{6};\ncross{2}{6};
				\ncross{8}{7};\ncross{7}{7};\ucross{5}{7};\ncross{4}{7};\ncross{3}{7};\ncross{2}{7};
				\ncross{8}{8};\ncross{7}{8};\ncross{6}{8};\ucross{4}{8}\ncross{3}{8};\ncross{2}{8};
				\ncross{8}{9};\ucross{6}{9};\ncross{5}{9};\ncross{4}{9}\ncross{3}{9};\ncross{2}{9};
				\ncross{8}{10};\ncross{7}{10};\ncross{6}{10};\rcross{4}{10}\ncross{3}{10};\ncross{2}{10};										\ncross{8}{11};\ncross{7}{11};\ucross{5}{11};\ncross{4}{11};\ncross{3}{11};\ncross{2}{11};
				\ncross{8}{12};\ocross{6}{12};\ncross{5}{12};\ncross{4}{12};\ncross{3}{12};\ncross{2}{12};
				\ncross{8}{13};\ncross{7}{13};\ocross{5}{13};\ocross{3}{13}\ncross{2}{13};
				\ncross{8}{14};\ncross{7}{14};\ncross{6}{14};\ocross{4}{14};\ncross{3}{14};\ncross{2}{14};
				\ncross{8}{15};\ncross{7}{15};\mcross{4}{3}{15};\ncross{3}{15};\ncross{2}{15};
				\ncross{8}{16};\ncross{7}{16};\ncross{6}{16};\ocross{4}{16};\ncross{3}{16};\ncross{2}{16};
				\ncross{8}{17};\ncross{7}{17};\ncross{6}{17};\ncross{5}{17};\ucross{3}{17};\ncross{2}{17};
				\mcross{6}{3}{18};\ncross{5}{18};\ncross{4}{18};\ncross{3}{18};\ncross{2}{18};
				\ncross{8}{19};\ncross{7}{19};\rcross{5}{19};\ncross{4}{19};\ncross{3}{19};\ncross{2}{19};
				\ncross{8}{20};\ncross{7}{20};\ncross{6}{20};\mcross{3}{3}{20};\ncross{2}{20};
				\ncross{8}{21};\ncross{7}{21};\ucross{5}{21};\ncross{4}{21};\ncross{3}{21};\ncross{2}{21};		
				
		\draw[thick,cap=round,line width=2pt] (1,5) -- (2,4);
		\draw[thick,cap=round,line width=2pt] (6,7) -- (7,6);
		\draw[thick,cap=round,line width=2pt] (8,4) -- (9,5);				
				
		\draw[thick,cap=corner,line width=6pt,color=white]	(1,5) -- (2,4);
		\draw[thick,cap=corner,line width=2pt,dotted]	(1,5) -- (2,4);
		
		\draw[thick,cap=corner,line width=6pt,color=white] (7,5) -- (8,6);
		\draw[thick,cap=corner,dotted] (7,5) -- (8,6);
		
		\draw[thick,cap=corner,line width=6pt,color=white]	(8,4) -- (9,5);
		\draw[thick,cap=corner,line width=2pt,dotted]	(8,4) -- (9,5);
		
	
		\node at (1.2,5) [above] {$\varsigma_1$};
		\node at (6.8,5) [above] {$\varsigma_2$};
		\node at (7.8,4) [above] {$\varsigma_3$};
				
	\end{scope}

\end{tikzpicture}
	\caption{Wiring diagram of positive MacLane's arrangement}
	\label{MacLaneWiring}
\end{figure}
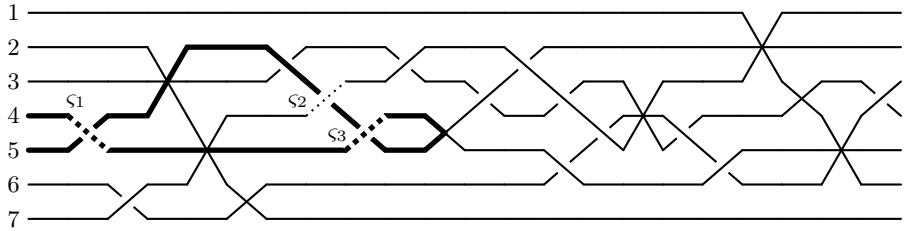

Let $W_{Q^+}$ be the braided wiring diagram of $Q^+$ given in Figure~\ref{MacLaneWiring}. Note that $W_{Q^+}$ differs from the wiring diagram considered 
 in \cite{cosu:08}  by an axial symmetry and a local move on the wires corresponding to $L_3,L_5,L_7$. 

The diagram  $W_{Q^+}$  is used to compute the unknotting map $\delta$,  and the images of the cycles $\epsilon$ in terms of geometric cycles, see Proposition \ref{image}. The thick lines in Figure \ref{MacLaneWiring} represent the cycle $\xi_{4,5}$, divided into two arcs of $L_4$ and $L_5$.

\begin{itemize}
	\item[-] The first arc  $L_4$  meets the triple point  $v_{P_{2,4,7}}$. This gives $\delta^l_{4,5}=\epsilon_{2,4}\inv	\alpha_2\inv\epsilon_{2,4}$.  
	\item[-] The second arc $L_5$ meets  $v_{P_{2,3,5}}$, and $\delta^r_{4,5}=\epsilon_{2,5}\inv	 (\epsilon_{2,3}\alpha_3 \epsilon_{2,3}\inv ) \alpha_2 \epsilon_{2,5}$.
\end{itemize}
This implies that
	\begin{equation*}\label{sigma45}
		\delta(\epsilon_{4,5}) = 
		\left(\epsilon_{2,4}\inv \alpha_2\inv\epsilon_{2,4}\right)\ .\ \epsilon_{4,5}\ .\
		\left[\epsilon_{2,5}\inv \left( \epsilon_{2,3}\alpha_3 \epsilon_{2,3}\inv \right) \alpha_2 \epsilon_{2,5}\right].
	\end{equation*}

Similarly, one computes:
\begin{align*}
	\delta(\epsilon_{2,3}) &= \epsilon_{2,3}, \\
	\delta(\epsilon_{2,5}) &=\epsilon_{2,5}, \\
	\delta(\epsilon_{2,4}) &=\epsilon_{2,4}, \\
	\delta(\epsilon_{2,7}) &=\epsilon_{2,7}, \\
	\delta(\epsilon_{2,6}) &=\epsilon_{2,6}, \\
	\delta(\epsilon_{4,5}) &=\left(\epsilon_{2,4}\inv \alpha_2\inv \epsilon_{2,4}\right)\ .\ \epsilon_{4,5}\ .\  \left[\epsilon_{2,5}\inv \left( \epsilon_{2,3} \alpha_3 \epsilon_{2,3}\inv \right) \alpha_2 \epsilon_{2,5}\right] \\
	\delta(\epsilon_{3,6}) &=\left(\epsilon_{2,3}\inv \alpha_2\inv \epsilon_{2,3}\right)\ .\ \epsilon_{3.6}\ .\ \left(\epsilon_{2,6}\inv \alpha_2 \epsilon_{2,6}\right) \\
	\delta(\epsilon_{3,7}) &=\left(\epsilon_{2,3}\inv \alpha_2\inv \epsilon_{2,3}\right)\ .\ \epsilon_{3,7}\ .\ \left[\epsilon_{2,7}\inv \left( \epsilon_{2,4} \alpha_4 \epsilon_{2,4}\inv \right) \alpha_2 \epsilon_{2,7} \right] \\
	\delta(\epsilon_{1,5}) &=\epsilon_{1,5}\ .\ \left[\left(\epsilon_{4,5}\inv \alpha_4 \epsilon_{4,5}\right) \left(\epsilon_{2,5}\inv \left( \epsilon_{2,3} \alpha_3 \epsilon_{2,3} \right) \inv \alpha_2 \epsilon_{2,5} \right) \right] \\
	\delta(\epsilon_{1,7}) &=\epsilon_{1,7}\ .\ \left[\left(\epsilon_{3,7}\inv \left(\epsilon_{3,6} \alpha_6 \epsilon_{3,6}\inv \right) \alpha_3 \epsilon_{3,7} \right) \left(\epsilon_{2,7}\inv \left(\epsilon_{2,4} \alpha_4 \epsilon_{2,4}\inv \right)\alpha_2 \epsilon_{2,7} \right) \right] \\
	\delta(\epsilon_{1,3}) &=\epsilon_{1,3}\ .\ \left(\epsilon_{2,3}\inv \alpha_2\epsilon_{2,3}\right) \\
	\delta(\epsilon_{1,4}) &=\epsilon_{1,4}\ .\ \left(\epsilon_{2,4}\inv \alpha_2\epsilon_{2,4}\right) \\
	\delta(\epsilon_{1,6}) &=\epsilon_{1,6}\ .\ \left[ \left(\epsilon_{3,6}\inv \alpha_3\epsilon_{3,6}\right) \left(\epsilon_{2,6}\inv \alpha_2 \epsilon_{2,6} \right) \right] 
\end{align*}


\subsection*{Retractions of geometric cycles}
\vspace{0.25cm}

We now compute the family of $\mu_{s,t}$, required to obtain the inclusion map, see Section~\ref{inclusion}.
The arcs of the wiring diagram $W_{Q^+}$ are labelled by the algorithm of W.~Arvola, see Section~\ref{color}.

The case of $\mu_{4,5}$ is drawn in thick in Figure \ref{MacLaneWiring}. The over arcs $\varsigma_1$, $\varsigma_2$ and $\varsigma_3$ are dotted in Figure \ref{MacLaneWiring}.  Arvola's labellings of these arcs are respectively : $a_{\varsigma_1}=\alpha_4$, $a_{\varsigma_2}=\alpha_7$ and $a_{\varsigma_3}=\alpha_7\inv \alpha_4 \alpha_7$. Furthermore, $\sgn(\varsigma_1)=-1$, $\sgn(\varsigma_2)=1$ and $\sgn(\varsigma_3)=1$. We obtain $\mu_{4,5} = \left( \alpha_7\inv \alpha_4 \alpha_7 \right)  \alpha_7  \alpha_4\inv $, which gives
\begin{equation*}
	\mu_{4,5} = \left( \alpha_7\inv \alpha_4 \alpha_7 \right)\ .\ \alpha_7\ .\ \alpha_4\inv.
\end{equation*}

Similarly:
\vspace{-5pt}
\begin{align*}
		\mu_{2,3}  &= 1, \\[-3pt]
		\mu_{2,5}  &= -\alpha_4, \\[-3pt]
		\mu_{2,4}  &= 1, \\[-3pt]
		\mu_{2,7}  &= 1, \\[-3pt]
		\mu_{2,6}  &= \alpha_7, \\[-3pt]
		\mu_{4,5}  &= \left( \alpha_7\inv \alpha_4 \alpha_7 \right)\ .\ \alpha_7\ .\ \alpha_4\inv, \\[-3pt]
		\mu_{3,6}  &= \left[\left(\alpha_4\inv\alpha_5\alpha_4\right)\left(\alpha_7\inv\right)\left(\alpha_7\inv\alpha_4\alpha_7^2\alpha_4\inv\ \alpha_5\inv\ \alpha_4\alpha_7^{-2}\alpha_4\inv\alpha_7\right)\left(\alpha_7\right)\right]\ .\\[-3pt]
		& \hspace{8cm} \left[\left(\alpha_7\inv\right)\left(\alpha_7\inv\alpha_4\inv\alpha_7\right)\left(\alpha_7\right)\right], \\[-3pt]
		\mu_{3,7}  &= \left[\left(\alpha_4\inv\alpha_5\alpha_4\right)\left(\alpha_7\inv\right)\left(\alpha_7\inv\alpha_4\alpha_7^2\alpha_4\inv\ \alpha_5\inv\ \alpha_4\alpha_7^{-2}\alpha_4\inv\alpha_7\right)\left(\alpha_7\right)\right], \\[-3pt]
		\mu_{1,5}  &=  \left(\alpha_7\inv\right)\left(\alpha_7\inv\alpha_4\alpha_7\right)\left(\alpha_7\right)\left(\alpha_4\inv\right), \\[-3pt]
		\mu_{1,7}  &= 1, \\[-3pt]
		\mu_{1,3}  &= \left(\alpha_7\inv\alpha_4\inv\alpha_7^2\ \alpha_6\inv\ \alpha_7^{-2}\alpha_4\alpha_7\right) \left(\alpha_7\inv\right)  \\[-3pt]
		& \hspace{4cm}\left(\alpha_7\inv\alpha_4\alpha_7^2\alpha_4\inv\ \alpha_5\ \alpha_4\alpha_7^{-2}\alpha_4\inv\alpha_7\right) \left(\alpha_7\right) \left(\alpha_4\inv\alpha_5\inv\alpha_4\right), \\[-3pt]
		\mu_{1,4}  &= 1, \\[-3pt]
		\mu_{1,6}  &= \left(\alpha_7\inv\alpha_4\alpha_7\right)\left(\alpha_7\inv\right)\left(\alpha_7\inv\alpha_4\inv\alpha_7\right)\left(\alpha_7\right). 
\end{align*}


\subsection*{Images in the group of the complement}\label{ImagesExamples}
\vspace{0.25cm}

Following Theorem \ref{MainResult}, we can compute $i_*: \pi_1(B_{Q^+} ) \twoheadrightarrow \pi_1(E_{Q^+})$. 
 The computations above describe the relations induced by the images of the cycles $\epsilon$ in $\pi_1(E_{Q^+})$. By the previous computations, $\epsilon_{2,3},\epsilon_{2,4},\epsilon_{2,7}$ are equal to $1$ (i.e. they are contractible in $E_{Q^+}$), and they are relations $r_{2,3},\ r_{2,4}$ and $r_{2,7}$. Without additional computation, we obtain:
\begin{equation*}
	\begin{array}{l p{2cm}l}
		r_{2,5}\ :\ \epsilon_{2,5} = \alpha_4\inv, &&
		r_{2,6}\ :\ \epsilon_{2,6} = \alpha_7.
	\end{array}
\end{equation*}

The case of $r_{4,5}$: 
\begin{equation*}
	r_{4,5}\ :\ (\epsilon_{2,4}\inv \alpha_2\inv\epsilon_{2,4})\ .\ \epsilon_{4,5}\ .\
	[\epsilon_{2,5}\inv ( \epsilon_{2,3}\alpha_3 \epsilon_{2,3}\inv ) \alpha_2 \epsilon_{2,5}]  =  \left( \alpha_7\inv \alpha_4 \alpha_7 \right)\ .\ \alpha_7\ .\ \alpha_4\inv.
\end{equation*}
Then using $r_{2,4}$, $r_{2,5}$ and $r_{2,3}$, we obtain that:
\begin{equation*}
	r_{4,5}\ :\ \epsilon_{4,5} = \left(\alpha_2\right)\ .\ \left(\left( \alpha_7\inv \alpha_4 \alpha_7 \right)\ .\ \alpha_7\ .\ \alpha_4\inv\right)\ .\ \left( \alpha_4 \alpha_2\inv \alpha_3\inv \alpha_4\inv \right).	
\end{equation*}

The others relations can be computed by the same way, and from the proof of Corollary \ref{MainCor}, we obtain:

\begin{propr}\label{Fundamental_Group_MacLane}
The fundamental group of $E_{Q^+}$ admits the following presentation:
	\begin{multline*}
		\p(E_{Q^+})= \langle \alpha_1,\alpha_2,\alpha_3,\alpha_4,\alpha_5,\alpha_6,\alpha_7,\\[-2pt]
		\epsilon_{2,3}, \epsilon_{2,5}, \epsilon_{2,4}, \epsilon_{2,7}, \epsilon_{2,6}, \epsilon_{4,5}, \epsilon_{3,6}, \epsilon_{3,7}, \epsilon_{1,5}, \epsilon_{1,7},\epsilon_{1,3}, \epsilon_{1,4},\epsilon_{1,6}\ |\\[-2pt]
		r_{2,3}, r_{2,5}, r_{2,4}, r_{2,7}, r_{2,6}, r_{4,5}, r_{3,6}, r_{3,7}, r_{1,5}, r_{1,7},r_{1,3}, r_{1,4},r_{1,6},\\
	[\alpha_7^{\epsilon_{1,7}},\alpha_5^{\epsilon_{1,5}},\alpha_1],
	[\alpha_3^{\epsilon_{1,3}},\alpha_1],
	[\alpha_6^{\epsilon_{1,6}},\alpha_4^{\epsilon_{1,4}},\alpha_1],
	[\alpha_5^{\epsilon_{2,5}},\alpha_3^{\epsilon_{2,3}},\alpha_2],\\
	[\alpha_7^{\epsilon_{2,7}},\alpha_4^{\epsilon_{2,4}},\alpha_2],
	[\alpha_6^{\epsilon_{2,6}},\alpha_2],
	[\alpha_7^{\epsilon_{3,7}},\alpha_6^{\epsilon_{3,6}},\alpha_3],
	[\alpha_5^{\epsilon_{4,5}},\alpha_4],
	\alpha_0\cdots\alpha_n
		\rangle.
	\end{multline*}
\end{propr}

\begin{acknowledgements}
	The authors thank E. Artal Bartolo for helpful comments and suggestions.
\end{acknowledgements}

\end{document}